\documentclass[a4paper,english,sumlimits,reqno]{amsart}

\usepackage[latin1]{inputenc}
\usepackage[T1]{fontenc}
\usepackage{babel}

\usepackage{color}
\usepackage{graphicx}

\usepackage{texdraw}

\usepackage{pstricks}
\usepackage{pstricks-add}
\usepackage{pst-node}

\usepackage{fp}

\usepackage[section]{placeins}

\usepackage{enumerate}

\usepackage{varioref}

\usepackage{amssymb}
\usepackage{amsmath}
\usepackage{amsthm}
\usepackage{mathtools}

\usepackage[mathcal]{eucal}
\usepackage{mathrsfs}

\theoremstyle{plain}

\newtheorem{thm}{Theorem}[section]
\newtheorem*{thm*}{Theorem}

\newtheorem{lem}[thm]{Lemma}
\theoremstyle{definition}

\theoremstyle{remark}

\newtheorem{rem}[thm]{Remark}

\numberwithin{equation}{section}

\newcommand{\bb}[1]{\ensuremath{\mathbb #1}}
\newcommand{\cal}[1]{\ensuremath{\mathcal #1}}
\newcommand{\scr}[1]{\ensuremath{\mathscr #1}}

\newcommand{\union}{\cup}
\newcommand{\Union}{\bigcup}
\newcommand{\isect}{\cap}

\newcommand{\umd}{\ensuremath{\mathrm{UMD}}}
\newcommand{\umdconst}{\ensuremath{\mathscr{U}}}

\newcommand{\charfun}{\ensuremath{\mathtt 1}}

\newcommand{\dint}{\scr D}

\DeclareMathOperator{\sign}{sign}
\DeclareMathOperator{\Id}{Id}

\DeclareMathOperator{\card}{\#}

\DeclareMathOperator{\supp}{supp}
\DeclareMathOperator{\dist}{dist}

\DeclareMathOperator{\cond}{\bb E}

\DeclareMathOperator{\indop}{\bb I}
\DeclareMathOperator{\salg}{\text{$\sigma$-algebra}}
\DeclareMathOperator{\pred}{\pi}

\title{The One--Third--Trick and Shift Operators}

\author{Richard Lechner}

\thanks{Partially supported by FWF P20166-N18}

\address{Institute of Analysis
  \newline
  \indent Johannes Kepler University Linz
  \newline
  \indent Altenbergerstrasse 69
  \newline
  \indent A-4040 Linz, Austria
}

\date{\today}

\begin{document}

\begin{abstract}
In this paper we obtain a representation as martingale transform
operators for the rearrangement and shift operators introduced by
T. Figiel in \cite{figiel:1988} and~\cite{figiel:1990}.
The martingale transforms and the underlying sigma algebras are
obtained explicitly by combinatorial means.
The known norm estimates for those operators are a direct consequence
of our representation.
\end{abstract}

\maketitle

\tableofcontents

\bigskip

\section{Introduction}\label{s:introduction}

The proof of the $T(1)$ theorem by T. Figiel proceeds by expanding the
integral operator into an absolutely converging series of basic
building blocks $T_m$ and $U_m$, rearranging and shifting the Haar
system.
This involves the following norm estimates for those building blocks,
which T. Figiel obtained by combinatorial means:
\begin{align}
  \|T_m\, :\, L_X^p \rightarrow L_X^p\|
  & \leq C\, \big( \log_2 (2 + |m|) \big)^\alpha,
  \label{eqn:intro-1}\\
  \|U_m\, :\, L_X^p \rightarrow L_X^p\|
  & \leq C\, \big( \log_2 (2 + |m|) \big)^\beta,
  \label{eqn:intro-2}
\end{align}
where the constant $C > 0$ depends only on $p$, the $\umd$--constant
of $X$ and $0 < \alpha, \beta < 1$.
For the original proof see \cite{figiel:1988}
and~\cite{figiel_wojtaszczyk:2001}.
See also~\cite{novikov_semenov:1997} and~\cite{mueller_p.f.x.:2005}.
For extensions to spaces of the homogeneous type
see~\cite{mueller_passenbrunner:2011}.

The purpose of the present paper is to obtain a representation of
$T_m$ and $U_m$ as the sum of roughly $\log_2(2+|m|)$ martingale
transform operators. This is done by combinatorial analysis of the
equations defining $T_m$ and $U_m$.

Our combinatorial analysis exhibits the link of T. Figiel's
rearrangement and shift operators to the so called one-third-trick
originating in the work of \cite{wolff:1982},
\cite{garnett_jones:1982}, \cite{davis:1980}
and~\cite{chang_wilson_wolff:1985}.

\subsection*{Related Recent Developments}

This article is taken from my Ph.D. thesis~\cite{lechner:2011}.

Recently T. P. Hyt\"onen, see~\cite{hytonen:2011}, presented his own
proof of of T. Figiel's vector-valued $T(1)$ theorem,
see~\cite{figiel:1990}.
The basic aim of T. P. Hyt\"onen in~\cite{hytonen:2011} is the same as
ours, to find reductions of the general case to certain preferable
situations where the so called Figiel compatibility condition is
satisfied.
To this end the problem in~\cite{hytonen:2011} is
randomized and the properties of the so called random dyadic
partitions of Nazarov, Treil and Volberg,
see~\cite{nazarov_treil_volberg:1997,nazarov_treil_volberg:2003} are
exploited.

By contrast the proof in the present paper proceeds by finding
explicitly those filtrations that turn a given bad combinatorial
situation into a good one, such that Figiel's compatibility condition
is satisfied. Our reduction is self-contained and develops specific
combinatorics of colored dyadic intervals.

\subsection*{Acknowledgment}

This article is part of my Ph.D. thesis~\cite{lechner:2011} written at
the Department of Analysis at the Johannes Kepler University Linz,
Austria.
I want to thank my Ph.D. advisor P. F. X. M\"uller for helpful
conversations.
This paper benefited greatly from extended discussions with Prof. Anna
Kamont.
\bigskip
\section{Preliminaries}\label{s:preliminaries}

\subsection*{The Haar System in $\bb R$}\hfill

Consider the collection of dyadic intervals at scale $j \in \bb Z$
given by
\begin{equation*}
  \dint_j
  = \big\{\ [2^{-j} k, 2^{-j}(k+1) [\, :\, k \in \bb Z \ \big\},
\end{equation*}
and the collection of the dyadic intervals
\begin{equation*}
  \dint = \Union_{j \in \bb Z} \dint_j.
\end{equation*}
We define the $L^\infty$--normalized Haar system by
\begin{align*}
  h_{[0,1[}(t)
  & = \charfun_{[0,\frac{1}{2}[}(t) - \charfun_{[\frac{1}{2},1[}(t),
  \qquad t \in \bb R,
  \intertext{and for every $I \in \dint$ set}
  h_I(t) & = h_{[0,1[}\big( \frac{t - \inf I}{|I|} \big),
  \qquad t \in \bb R,
\end{align*}
where $\charfun_A$ denotes the characteristic function of a set $A$.

\subsection*{Banach Spaces with the $ \umd$--Property}\hfill

By $L^p(\Omega, \mu; X)$ we denote the space of functions with values
in $X$, Bochner--integrable with respect to $\mu$.
If $\Omega = \bb R$ and $\mu$ is the Lebesgue measure $|\cdot|$ on
$\bb R$, then set $L_X^p(\bb R) = L^p(\bb R, |\cdot|; X)$, if
unambiguous further abbreviated as $L_X^p$.

\smallskip

We say $X$ is a $\umd$ space if for every $X$--valued martingale
difference sequence $\{d_j\}_j \subset L^p(\Omega,\mu;X)$, $1 < p < \infty$, and
choice of signs $\varepsilon_j \in \{-1,1\}$ one has
\begin{equation}\label{eqn:umd--property}
  \big\| \sum_j \varepsilon_j\, d_j \big\|_{L^p(\Omega,\mu;X)}
  \leq \umdconst_p(X)\cdot \big\| \sum_j d_j \big\|_{L^p(\Omega,\mu;X)},
\end{equation}
where $\scr U_p(X)$ does not depend on $\varepsilon_j$ or $d_j$.
The constant $\scr U_p(X)$ is called $\umd$--constant.
We refer the reader to~\cite{burkholder:1981}.

\subsection*{Kahane's Contraction Principle}\hfill

For every Banach space $X$, $1 \leq p < \infty$, finite set
$\{x_j\}_j \subset X$ and bounded sequence of scalars $\{c_j\}_j$ we
have
\begin{equation}\label{eqn:kahanes_contraction_principle}
  \bigg(
  \int_0^1 \Big\| \sum_j r_j(t)\, c_j\, x_j \Big\|_X^p\, \mathrm d t
  \bigg)^{1/p}
  \leq \sup_j |c_j|\cdot
  \bigg(
  \int_0^1 \Big\| \sum_j r_j(t)\, x_j \Big\|_X^p\, \mathrm d t
  \bigg)^{1/p},
\end{equation}
where $\{r_j\}_j$ denotes an independent sequence of Rademacher
functions.
For details see~\cite{kahane:1985}.
Below we give a short proof, see~\cite{kahane:1985}.
\begin{proof}
  By scaling inequality~\eqref{eqn:kahanes_contraction_principle}, we
  may assume $|c_j| \leq 1$, for all $j$.
  We represent each $c_j$ as the series
  $c_j = \sum_{k\geq 1} \varepsilon_{jk}\, 2^{-k}$, with suitable
  $\varepsilon_{jk} \in \{\pm 1\}$ and observe
  \begin{align*}
    \bigg(
    \int_0^1 \Big\| \sum_j r_j(t)\, c_j\, x_j \Big\|_X^p\, \mathrm d t
    \bigg)^{1/p}
    & \leq \sum_{k\geq 1} 2^{-k}
    \bigg(
    \int_0^1 \Big\| \sum_j r_j(t)\, \varepsilon_{jk}\, x_j
    \Big\|_X^p\, \mathrm d t
    \bigg)^{1/p}\\
    & = \bigg(
    \int_0^1 \Big\| \sum_j r_j(t)\, x_j \Big\|_X^p\, \mathrm d t
    \bigg)^{1/p}.
  \end{align*}
  The last equality holds true since
  $\sum_j r_j(t)\, \varepsilon_{jk}\, x_j$ has the same distribution
  as $\sum_j r_j(t)\, x_j$ for all choices of signs $\varepsilon_{jk}$.
\end{proof}

\subsection*{The Martingale Inequality of Stein -- Bourgain's Version}\hfill

Let $(\Omega,\cal F,\mu)$ be a probability space,
and let $\cal F_1 \subset \ldots \subset \cal F_m \subset \cal F$
denote an increasing sequence of $\sigma$--algebras.
For every choice of $f_1,\ldots,f_m \in L^p(\Omega,\mu;X)$, $1 < p < \infty$,
let $r_1,\ldots,r_m$ denote independent Rademacher functions, then
\begin{equation}\label{eqn:steins_martingale_inequality}
  \int_0^1 \big\|
    \sum_{i=1}^m r_i(t)\, \cond(f_i | \cal F_i)
  \big\|_{L^p(\Omega,\mu;X)}\, \mathrm dt
  \leq C\cdot   \int_0^1 \big\|
    \sum_{i=1}^m r_i(t)\, f_i
  \big\|_{L^p(\Omega,\mu;X)}\, \mathrm dt,
\end{equation}
where $C$ depends only on $p$ and $\umdconst_p(X)$.
The scalar valued version of~\eqref{eqn:steins_martingale_inequality}
by E. M. Stein can be found in~\cite{stein:1970-lp}.
The vector valued extension is due to
J. Bourgain~\cite{bourgain:1986}.
A proof may be found in~\cite{figiel_wojtaszczyk:2001}.

\subsection*{Additional Notation}\hfill

Let $\scr N$ be a collection of nested sets, then
$\pred_{\scr N}\, :\, \scr N \rightarrow \scr N$ is defined as follows.
Let $K \in \scr N$ and then $\pred_{\scr N}(K)$ is the smallest element with
respect to inclusion of the
collection $\{M \in \scr N\, :\, K \subsetneq M\}$.
In most cases we will omit the subscript and explicitly state to which nested
collection we refer.

Given a collection of Lebesgue measurable sets $\scr L$, the collection of
Lebesgue measurable sets $\salg(\scr L)$ denotes
the smallest sigma algebra containing $\scr L$.
\bigskip
\section{The One--Third--Trick}\label{s:one-third-trick}

We will introduce and investigate two variants of one-third-shift
operators, that is the bilateral alternating one-third-shift operator
and the unilateral one-third-shift operator.
First we will introduce the bilateral alternating one-third-shift
operator $S$ given by $S(h_I) = h_{\sigma(I)}$,
see~\eqref{eqn:one-third-shift-operator}.
Roughly speaking, $\sigma$ shifts intervals, say for example having
length $1$, to the \emph{right} by one third of their length, so in our
instance by $\frac{1}{3}$.
The intervals having length $\frac{1}{2}$ are then shifted by one
third of their size to the \emph{left}, so by $\frac{1}{6}$.
Hence the relative translation of two successive levels
of dyadic intervals amounts to a total of $\frac{1}{2}$, thus yielding
a nested collection of intervals, again.
This is illustrated in Figure~\vref{fig:shifted_intervals}.
In Theorem~\ref{thm:one-third-trick-isomorphism} we establish that
$S\, :\, L_X^p \longrightarrow L_X^p$ is an isomorphism by means of
Bourgain's version of Stein's martingale inequality.

Finally, we will consider the unilateral variants $S_0$ and $S_1$ of
the one-third-shift operator, and establish in
Theorem~\ref{thm:modified-one-third-trick-isomorphism} that both are
isomorphic maps from $L_X^p$ to itself, as well.

The one-third-trick originates with the work of~\cite{wolff:1982},
\cite{garnett_jones:1982} and~\cite{chang_wilson_wolff:1985}.

\subsection{Bilateral Alternating One-Third-Shift}\label{ss:alternating}\hfill

For every $j \in \bb Z$ let
\begin{equation}\label{eqn:shift_width--1}
  s_j = (-1)^j\,  2^{-j}/3,
\end{equation}
and define
\begin{equation}\label{eqn:shift_width--2}
  s(I) = s_j,
\end{equation}
for all intervals $I$ having measure $|I| = 2^{-j}$.
Then define the one--third--shift map
\begin{equation}\label{eqn:one-third-shift-map}
  \sigma(I) = I + s(I),
\end{equation}
and the one--third--shift operator
\begin{equation}\label{eqn:one-third-shift-operator}
  S(h_I) = h_{\sigma(I)},
\end{equation}
where by $h_{\sigma(I)}$ we denote the function
$h_{\sigma(I)}(x) = h_I(x - s(I))$.
The one--third--shift of dyadic intervals for two consecutive
levels is illustrated in Figure~\vref{fig:shifted_intervals}.
\begin{figure}[bt]
  \begin{center}
    \begin{pspicture}(-2,1)(10.5,-4)
      \newcommand{\interval}[1]{%
        \psline[linewidth=1.2pt,tbarsize=0.15]{|*-|*}(#1,0)}

      \multirput(0,0)(4.5,0){2}{%
        \multirput(0,0)(1.5,-1){2}{%
          \interval{4.5}}}
      \multirput(0,-3)(2.25,0){4}{%
        \multirput(0,0)(-.75,1){2}{%
          \interval{2.25}}}
      \psline[linestyle=dotted](0,-3)(0,0)
      \psline[linestyle=dotted](4.5,-3)(4.5,0)
      \psline[linestyle=dotted](9,-3)(9,0)
      \psline[linestyle=dotted](1.5,-2)(1.5,-1)
      \psline[linestyle=dotted](6,-2)(6,-1)
      \pnode(0,0){coarse-left}
      \pnode(1.5,-1){coarse-right}
      \ncline[nodesep=.25]{->}{coarse-left}{coarse-right}
      \nbput{$s_j$}
      \pnode(8.25,-2){coarse-left}
      \pnode(9,-3){coarse-right}
      \ncline[nodesep=.25]{<-}{coarse-left}{coarse-right}
      \naput*{$s_{j+1}$}
      \rput(-2,0){%
        \uput[0](0,-.5){Level $j$}
        \uput[0](0,-2.5){Level $j+1$}}
    \end{pspicture}
  \end{center}
  \caption{One--third--shift of two consecutive levels of
    intervals. In this illustration $j$ is even.}
  \label{fig:shifted_intervals}
\end{figure}
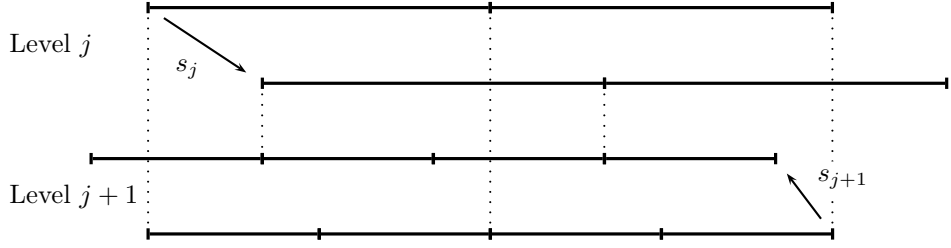

From this picture one can see that the collection of
one--third--shifted dyadic intervals $\sigma(\dint)$ is nested, and
$\dint \isect \sigma(\dint) = \emptyset$.
Note that if a one--third--shifted dyadic interval
$J \in \sigma(\dint)$ is contained in a non--shifted interval
$I \in \dint$, then $\dist(J,I^c) \geq |J|/3$.
For every given interval $I \in \dint$ exists a unique
one--third--shifted interval $J \in \sigma(\dint)$, $|J| = |I|/2$
being contained in $I$.
First observe that for every $j \in \bb Z$ and $I \in \dint_j$ we have
\begin{equation*}
  \begin{aligned}
    \card \big\{
    J \in \sigma ( \dint_{j+1} )\, :\,
    J \isect I \neq \emptyset
    \big\} & = 3,\\
    \card \big\{
    J \in \sigma ( \dint_{j+1} )\, :\,
    J \subset I
    \big\} & = 1.
  \end{aligned}
\end{equation*}
So we can define $\omega(I)$ by
\begin{equation}\label{eqn:associate}
  \omega(I) = J,
  \quad \text{where $J \in \sigma(\dint)$, $|J| = |I|/2$ and
    $J \subset I$},
\end{equation}
see Figure~\vref{fig:associate}.
\begin{figure}[bt]
  \begin{center}
    \begin{pspicture}(-.75,.5)(6.75,-4)
      \newcommand{\interval}[1]{%
        \psline[linewidth=1.2pt,tbarsize=0.15]{|*-|*}(#1,0)}

      \interval{4.5}
      \multirput(0,-1)(2.25,0){3}{\interval{2.25}}
      \multirput(-.75,-2)(2.25,0){3}{\interval{2.25}}
      \pnode(2.25,0){top}
      \pnode(2.625,-2){bottom}
      \ncline{->}{top}{bottom}
      \pnode(1.125,-1){left-top}
      \pnode(3.375,-1){middle-top}
      \pnode(5.625,-1){right-top}
      \rput(-.75,-1){%
        \pnode(1.125,-1){left-bottom}
        \pnode(3.375,-1){middle-bottom}
        \pnode(5.625,-1){right-bottom}}
      \ncline[linestyle=dashed,dash=2pt 1pt]{->}{left-top}{left-bottom}
      \ncput*{$\sigma$}
      \ncline[linestyle=dashed,dash=2pt 1pt]{->}{middle-top}{middle-bottom}
      \ncput*{$\sigma$}
      \ncline[linestyle=dashed,dash=2pt 1pt]{->}{right-top}{right-bottom}
      \ncput*{$\sigma$}
      \pcline[linewidth=.6pt,offset=-1cm,arrowsize=.1,tbarsize=.15]{|<*->|*}
      (0,-2)(1.5,-2)
      \nbput{$\frac{1}{3}\, |I|$}
      \pcline[linewidth=.6pt,offset=-1cm,arrowsize=.1,tbarsize=.15]{|<*->|*}
      (1.5,-2)(3.75,-2)
      \nbput{$\frac{1}{2}\, |I|$}
      \pcline[linewidth=.6pt,offset=-1cm,arrowsize=.1,tbarsize=.15]{|<*->|*}
      (3.75,-2)(4.5,-2)
      \nbput{$\frac{1}{6}\, |I|$}
      \uput[90](2.25,0){$I$}
      \uput[-90](2.625,-2){$\omega(I)$}
      \psline[linestyle=dotted](0,-3)(0,0)
      \psline[linestyle=dotted](1.5,-3)(1.5,-2)
      \psline[linestyle=dotted](3.75,-3)(3.75,-2)
      \psline[linestyle=dotted](4.5,-3)(4.5,0)
    \end{pspicture}
  \end{center}
  \caption{The interval $I$ has measure $|I| = 2^{-j}$ with $j$ being
    even.}
  \label{fig:associate}
\end{figure}
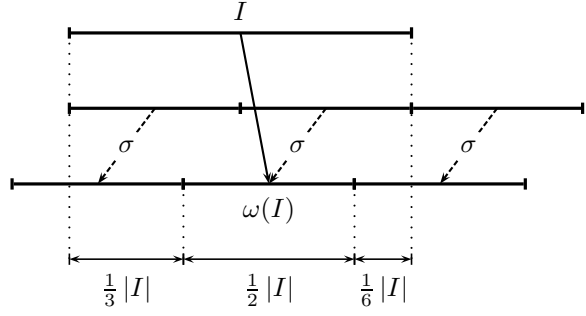

Note the basic properties summarized in
\begin{lem}\label{lem:one-third-shift}
  The following statements are true.
  \begin{enumerate}[(i)]
  \item $\sigma(\dint)$ is a nested collection of dyadic intervals,
    and $\dint \isect \sigma(\dint) = \emptyset$.
    \label{enu:one-third-shift-1}
  \item $\omega \, :\, \dint \longrightarrow \sigma(\dint)$ is well
    defined and injective.
    \label{enu:one-third-shift-2}
  \item Let $I \in \dint$, then $\omega(I) \subset I$.
    \label{enu:one-third-shift-3}
  \item For every $I \in \dint$ we have
    $\dist(\omega(I),I^c) = |I|/6$.
    \label{enu:one-third-shift-4}
  \item Let $I, J \in \dint$, $|I| = |J|$, then
    $\dist(\omega(I),\omega(J)) < |\omega(I)|$
    if and only if $I = J$.
    \label{enu:one-third-shift-5}
  \item For all $I \in \dint$ we have the identity
    $\sigma(I)
    = \omega(I) \union \big(
      \omega(I) + \sign(s(I)) \cdot |\omega(I)|
      \big)$.
    \label{enu:one-third-shift-6}
  \end{enumerate}
\end{lem}

\begin{proof}
  The assertions are easily verified.
\end{proof}

We need to build up some more notation. For all $j \in \bb Z$ and
\begin{equation*}
  u = \sum_{I \in \dint} u_I\, h_I\, |I|^{-1}
\end{equation*}
let $( u )_j$ restrict the function $u$ to level $j$, precisely
\begin{equation}\label{eqn:restriction}
  (u)_j = \sum_{I \in \dint_j} u_I\, h_I\, |I|^{-1}.
\end{equation}
If we define
\begin{equation}\label{eqn:indop}
  \indop (u)_j = \sum_{I \in \dint_j} u_I\, \charfun_I\, |I|^{-1},
\end{equation}
then we find due to Kahane's contraction
principle~\eqref{eqn:kahanes_contraction_principle} that
\begin{equation}\label{eqn:indop-isometry}
  \int_0^1
  \Big\| \sum_{j \in \bb Z} r_j(t)\, (u)_j \Big\|_{L_X^p}
  \mathrm dt
  = \int_0^1
  \Big\| \sum_{j \in \bb Z} r_j(t)\, \indop(u)_j \Big\|_{L_X^p}
  \mathrm dt.
\end{equation}

The following theorem establishes that the one--third--shift operator
$S\, :\, L_X^p \longrightarrow L_X^p$ is an isomorphism.
\begin{thm}\label{thm:one-third-trick-isomorphism}
  Let $1 < p < \infty$ and $X$ a Banach space with the
  $\umd$--property, then there exists a constant $C > 0$ such that
  \begin{equation*}
    \frac{1}{C}\, \big\| u \big\|_{L_X^p}
    \leq \big\| S u \big\|_{L_X^p}
    \leq C\, \big\| u \big\|_{L_X^p},
  \end{equation*}
  for all $u \in L_X^p$.
  The constant $C$ depends only on $\umdconst_p(X)$.
\end{thm}

\begin{proof}
  Let $u = \sum_{I \in \dint} u_I\, h_I\, |I|^{-1} \in L_X^p$ be
  fixed throughout this proof and set
  \begin{equation*}
    v = \sum_{I \in \dint}
    u_I\, h_{\omega(I)}\, |\omega(I)|^{-1}.
  \end{equation*}
  Note that $\{\omega(I)\, :\, I \in \dint\}$ is nested, see
  Lemma~\ref{lem:one-third-shift},
  assertion~\eqref{enu:one-third-shift-1}
  and~\eqref{enu:one-third-shift-2}.
  Observe we have due to Lemma~\ref{lem:one-third-shift},
  assertion~\eqref{enu:one-third-shift-3} that
  $\indop (u)_j = \cond (\indop(v)_j | \dint_j)$, so the
  $\umd$--property and Kahane's contraction
  principle~\eqref{eqn:kahanes_contraction_principle} yield
  \begin{align*}
    \|u\|_{L_X^p}
    & \lesssim \int_0^1
    \|\sum_{j \in \bb Z} r_j(t)\, (u)_j \|_{L_X^p}\, \mathrm dt\\
    & = \int_0^1
    \|\sum_{j \in \bb Z} r_j(t)\, \indop (u)_j \|_{L_X^p}\,
    \mathrm dt\\
    & = \int_0^1 \|
    \sum_{j \in \bb Z} r_j(t)\, \cond (\indop (v)_j | \dint_j)
    \|_{L_X^p}\, \mathrm dt.
  \end{align*}
  Now we apply Stein's martingale
  inequality~\eqref{eqn:steins_martingale_inequality} followed by
  identity~\eqref{eqn:indop-isometry} to pass from $\indop (v)_j$ to
  $(v)_j$, so
  \begin{align*}
    \int_0^1 \|
    \sum_{j \in \bb Z} r_j(t)\, \cond (\indop (v)_j | \dint_j)
    \|_{L_X^p}\, \mathrm dt
    & \lesssim \int_0^1 \|
    \sum_{j \in \bb Z} r_j(t)\, \indop (v)_j
    \|_{L_X^p}\, \mathrm dt\\
    & = \int_0^1 \|
    \sum_{j \in \bb Z} r_j(t)\, (v)_j
    \|_{L_X^p}\, \mathrm dt.
  \end{align*}
  Recalling definition~\eqref{eqn:one-third-shift-operator} 
  and applying Kahane's contraction principle in consideration of
  $\omega(I) \subset \sigma(I)$
  (see identity~\eqref{enu:one-third-shift-6} in
  Lemma~\ref{lem:one-third-shift}),
  we estimate
  \begin{equation*}
    \int_0^1 \|
    \sum_{j \in \bb Z} r_j(t)\, (v)_j
    \|_{L_X^p}\, \mathrm dt  
    \leq 2\cdot \int_0^1 \|
    \sum_{j \in \bb Z} r_j(t)\, (S u)_j
    \|_{L_X^p}\, \mathrm dt,
  \end{equation*}
  and the $\umd$--property implies
  \begin{equation*}
    \int_0^1 \|
    \sum_{j \in \bb Z} r_j(t)\, (v)_j
    \|_{L_X^p}\, \mathrm dt  
    \lesssim
    \|S u\|_{L_X^p}.
  \end{equation*}
  Thus, collecting the inequalities yields
  \begin{equation*}
    \| u \|_{L_X^p} \lesssim \| S u \|_{L_X^p}.
  \end{equation*}

  One can repeat the preceding argument with the roles of $u$ and
  $S u$ interchanged and obtain the converse inequality
  \begin{equation*}
    \|S u\|_{L_X^p} \lesssim \|u\|_{L_X^p}.
  \end{equation*}
\end{proof}

\subsection{Unilateral One-Third-Shift}\label{ss:unilateral}\hfill

We now introduce modified versions $\sigma_0$ and $\sigma_1$ of
the one-third-shift map $\sigma$.
To this end we define
$\sigma_0, \sigma_1\, :\, \dint \longrightarrow \sigma(\dint)$,
\begin{align}
  \sigma_0(I) &= J,
  \qquad\text{
    where $J \in \sigma(\dint)$, $|J| = |I|$ and $\sup J \in I$},
  \label{eqn:unilateral-map-0}\\
  \sigma_1(I) & = J,
  \qquad\text{
    where $J \in \sigma(\dint)$, $|J| = |I|$ and $\inf J \in I$},
  \label{eqn:unilateral-map-1}
\end{align}
see Figure~\vref{fig:unilateral_one-third-shift}.
This induces the one-third-shift operators $S_0$ and $S_1$ given by
the linear extension of
\begin{align}
  S_0(h_I) & = h_{\sigma_0(I)}, \qquad I \in \dint,
  \label{eqn:unilateral-operator-0}\\
  S_1(h_I) & = h_{\sigma_1(I)}, \qquad I \in \dint.
  \label{eqn:unilateral-operator-1}
\end{align}
\begin{figure}[bt]
  \begin{center}
    \begin{pspicture}(-2,.5)(10.5,-2.5)

      \newcommand{\interval}[1]{%
        \psline[linewidth=1.2pt,tbarsize=0.15]{|*-|*}(#1,0)}

      \rput(3,0){\interval{4.5}}
      \rput(0,-2){\multirput(4.5,0){2}{\interval{4.5}}}
      \pnode(5.25,0){orig}
      \pnode(2.25,-2){left}
      \pnode(6.75,-2){right}
      \ncline[nodesep=.25]{->}{orig}{left}
      \ncput*{$\sigma_0$}
      \ncline[nodesep=.25]{->}{orig}{right}
      \ncput*{$\sigma_1$}
      \uput[90](5.25,0){$I$}
      \uput[-90](2.25,-2){$\sigma_0(I)$}
      \uput[-90](6.75,-2){$\sigma_1(I)$}
    \end{pspicture}
  \end{center}
  \caption{Unilateral one-third-shifts $\sigma_0$ and $\sigma_1$
    applied to $I \in \dint$.
    In this picture the one-third-shift map $\sigma$ shifts to the
    right, so $\sigma_1(I) = \sigma(I)$.}
  \label{fig:unilateral_one-third-shift}
\end{figure}
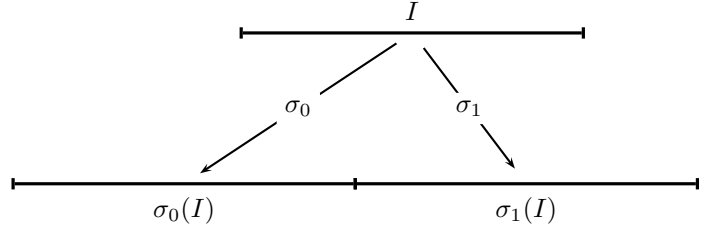

Observe that we have either
\begin{equation*}
  \sigma(I) = \sigma_0(I)
  \qquad \text{or}\qquad
  \sigma(I) = \sigma_1(I),
\end{equation*}
depending on the direction in which $\sigma$ one-third-shifts the
interval $I$.
Anyhow, we can see that
\begin{align*}
  |I \isect \sigma_0(I)| & \geq \frac{1}{3}\, |I|, &
  |I \isect \sigma_1(I)| & \geq \frac{1}{3}\, |I|,
\end{align*}
for all $I \in \dint$.
The proof of Theorem~\vref{thm:one-third-trick-isomorphism} with minor
modifications yields
Theorem~\ref{thm:modified-one-third-trick-isomorphism} below.
\begin{thm}\label{thm:modified-one-third-trick-isomorphism}
  Let $1 < p < \infty$ and $X$ a Banach space with the
  $\umd$--property, then there exists a constant $C > 0$ such that
  \begin{align*}
    \frac{1}{C}\, \big\| u \big\|_{L_X^p}
    \leq \big\| S_0 & u \big\|_{L_X^p}
    \leq C\, \big\| u \big\|_{L_X^p},\\
    \frac{1}{C}\, \big\| u \big\|_{L_X^p}
    \leq \big\| S_1 & u \big\|_{L_X^p}
    \leq C\, \big\| u \big\|_{L_X^p},
  \end{align*}
  for all $u \in L_X^p$.
  The constant $C$ depends only on $\umdconst_p(X)$.
\end{thm}

\begin{proof}
  Define $\omega_0$ and $\omega_1$ by
  \begin{align*}
    \omega_0(I) & = J,
    \qquad \text{where $J \in \sigma(\dint)$, $|J| = |I|/4$ and
      $\sup J = \sup \sigma_0(I)$},\\
    \omega_1(I) & = J,
    \qquad \text{where $J \in \sigma(\dint)$, $|J| = |I|/4$ and
      $\inf J = \inf \sigma_1(I)$},
  \end{align*}
  for all $I \in \dint$.
  Now all we need to do is repeat the proof of
  Theorem~\ref{thm:modified-one-third-trick-isomorphism} with $\omega$
  replaced by $\omega_\delta$ in order to estimate
  $S_\delta$, for each $\delta \in \{0,1\}$.
\end{proof}
\bigskip
\section{The Shift Operator $T_m$}\label{s:shift-tm}

Here we define $16 + 4\cdot \log_2(|m|)$, $m \neq 0$ collections of
the Haar system, so that on each such subcollection $T_m$ acts as a
martingale transform operator on either the dyadic grid or the
one-third-shifted dyadic grid.
In section~\ref{s:one-third-trick} we established that changing the dyadic grid
to the one-third-shifted dyadic grid is an isomorphism.
Thus we may assume that $T_m$ is representable as a martingale transform
operator on each of the $16 + 4\cdot \log_2(|m|)$ subcollections, which yields
the well known estimate
\begin{equation*}
\|T_m\, :\, L_X^p \rightarrow L_X^p\|
\leq C\cdot \big( \log_2 (2 + |m|) \big)^\alpha,
\end{equation*}
for some $0 < \alpha < 1$, established by T. Figiel in~\cite{figiel:1988}.

\bigskip

Define the shift map $\tau_m$, $m \in \bb Z$  by
\begin{equation}\label{eqn:shift-taum}
  \tau_m(I) = I + m\, |I|,
\end{equation}
for all $I \in \dint \union \sigma(\dint)$.
This induces the shift operator $T_m$, given by
\begin{equation}\label{eqn:shift-tm}
  T_m h_I = h_{\tau_m(I)},
\end{equation}
for all $I \in \dint \union \sigma(\dint)$.
It is crucial that the one--third--shift operator $S$ defined
in~\eqref{eqn:one-third-shift-operator} and the shift operator
$T_m$ commute, that is the identity
\begin{equation}\label{eqn:one-third-shift_and_tm_commute}
  (S \circ T_m) (u) = (T_m \circ S) (u),
\end{equation}
for all $u \in L_X^p$.
Analogously, we have that
\begin{align}
  (S_0 \circ T_m) (u) & = (T_m \circ S_0) (u),
  \label{eqn:unilateral_one-third-shift_and_tm_commute-0}\\
  (S_1 \circ T_m) (u) & = (T_m \circ S_1) (u),
  \label{eqn:unilateral_one-third-shift_and_tm_commute-1}
\end{align}
for all $u \in L_X^p$, see~\eqref{eqn:unilateral-map-0},
\eqref{eqn:unilateral-map-1}, \eqref{eqn:unilateral-operator-0} and
\eqref{eqn:unilateral-operator-1}.

We aim at splitting the dyadic intervals $\dint$ into collections
$\scr B_i^{(\delta)}$, such that we may bound $T_m\circ S^\delta$ on
functions supported on
$\sigma^\delta\big(\scr B_i^{(\delta)}\big)$, $\delta \in \{0,1\}$.
Note that if $\delta = 0$, then  $S^\delta = \Id$ and
$\sigma^\delta = \Id$.

\bigskip

Given a shift width $m \in \bb Z$, $m \neq 0$, we will partition the dyadic
intervals $\dint$ into $16 + 4 \cdot \log_2(|m|)$ disjoint collections denoted
by $\scr B_i^{(\delta)}$.
The collections are constructed in such way that for each
$i$ and $\delta \in \{0,1\}$ fixed, we have that whenever
$I \in \scr B_i^{(\delta)}$, the intervals $\sigma^\delta(I)$ and
$\big( \tau_m \circ \sigma^\delta\big)(I)$ share the same dyadic
predecessor with respect to the collection
$\sigma^\delta\big(\scr B_i^{(\delta)}\big)$.
The details are elaborated in Lemma~\ref{lem:shift-lemma} below.
\begin{lem}\label{lem:shift-lemma}
  For every integer $m \in \bb Z$, $m \neq 0$ let $\tau_m$ denote the
  map given by
  \begin{equation*}
    \tau_m(I) = I + m\, |I|,
  \end{equation*}
  for all $I \in \dint \union \sigma(\dint)$,
  see~\eqref{eqn:shift-taum}.

  Then there exist a constant $K(m) \leq 7 + 2\cdot \log_2(|m|)$ and
  disjoint collections of dyadic intervals
  $\scr B_i^{(\delta)}$, $0 \leq i \leq K(m)$, $\delta \in \{0, 1\}$
  with
  \begin{equation*}
    \dint
    = \Union_{\delta \in \{0, 1\}} \Union_{i = 0}^{K(m)}
    \scr B_i^{(\delta)},
  \end{equation*}
  such that
  \begin{equation}\label{eqn:shift-lemma-nested_collection}
    \big\{ I,\, \tau_m(I),\, I \union \tau_m(I)\, :\,
    I \in \sigma^{\delta}(\scr B_i^{(\delta)}) \big\}
  \end{equation}
  is a nested collection of sets, for all $0 \leq i \leq K(m)$ and
  $\delta \in \{0, 1\}$.
\end{lem}

\begin{proof}
  Due to symmetry we may assume that $m \geq 1$, and we set
  $K(m)=K(-m)$, if $m \leq -1$. So fix a shift width $m \geq 2$ and a
  $\lambda \geq 4$ such that
  \begin{equation}\label{eqn:dyadic_shift_width}
    2^{\lambda-3} \leq m < 2^{\lambda-2},
  \end{equation}
  and define $L(m) = \lambda-1$.
  If $m = 1$, then let $\lambda = 4$ and set $L(1) = 3$.
  Now we split $\dint$ into disjoint collections
  $\scr A_i$, $0 \leq i \leq L(m)$, by omitting $L(m)$ consecutive
  levels of $\dint$.
  More precisely, for every $0 \leq i \leq L(m)$ we define
  \begin{equation}\label{eqn:splitting_sets}
    \scr A_i = \Union_{j \in \bb Z}
      \big\{ I \in \dint\, :\,
        |I| = 2^{-(\lambda \cdot j + i)}
      \big\}.
  \end{equation}

  Next we want to divide each of the $\scr A_i$ into two collections
  $\scr A_i^{(0)}$ and $\scr A_i^{(1)}$, such that
  every $I \in \scr A_i^{(0)}$ has the same predecessor in
  $\scr A_i^{(0)}$ as $\tau_m(I)$, and $\scr A_i^{(0)}$ is maximal.
  As a consequence, the collection $\scr A_i^{(1)}$ consists all
  intervals $I$ such that $I$ and $\tau_m(I)$ do \emph{not} share the same
  predecessor. But, if we apply the one--third--shift map $\sigma$ to
  the collection $\scr A_i^{(1)}$, then every
  $I \in \sigma \big( \scr A_i^{(1)} \big)$ has the same predecessor
  in $\sigma \big( \scr A_i^{(1)} \big)$ as $\tau_m(I)$.
  We will now construct these two collections.
  To this end let $\scr G$ denote one of the collections
  $\scr A_i$, $\sigma\big( \scr A_i \big)$, $0 \leq i \leq L(m)$
  and define
  \begin{equation}\label{eqn:lambda_support-1}
    \begin{aligned}
      \scr C_0(\scr G, I)
      & = \big\{
        J \in \scr G\, :\,
        |J| = 2^{-\lambda}\, |I|,\
        J \subset I \text{ and } \tau_m(J) \subset I
      \big\},\\
      \scr C_1(\scr G, I)
      & = \big\{
        J \in \scr G\, :\,
        |J| = 2^{-\lambda}\, |I|,\
        J \subset I \text{ and } \tau_m(J) \isect I = \emptyset
      \big\}.
    \end{aligned}
  \end{equation}
  Revisiting the definition of the one--third--shift
  map~\eqref{eqn:one-third-shift-map} and considering the
  restriction~\eqref{eqn:dyadic_shift_width} one can see that
  \begin{equation}\label{eqn:splitting_sets--shifted-1}
    \sigma \Big( \scr C_1(\scr A_i, I) \Big)
    \subset \scr C_0 \Big( \sigma(\scr A_i), \sigma(I) \Big),
  \end{equation}
  for all $I \in \scr A_i$, $0 \leq i \leq L(m)$.
  This means that all intervals
  $J \in \scr \sigma\big(\scr C_1(\scr A_i, I)\big)$ are such
  that $J$ and $\tau_m(J)$ share $\sigma(I)$ as common predecessor
  with respect to the collection $\sigma \big( \scr A_i^{(1)} \big)$.
  In Figure~\vref{fig:basis_exchange} one can see the action of the
  one--third--shift map $\sigma$ on the collection $\scr A_i$.
  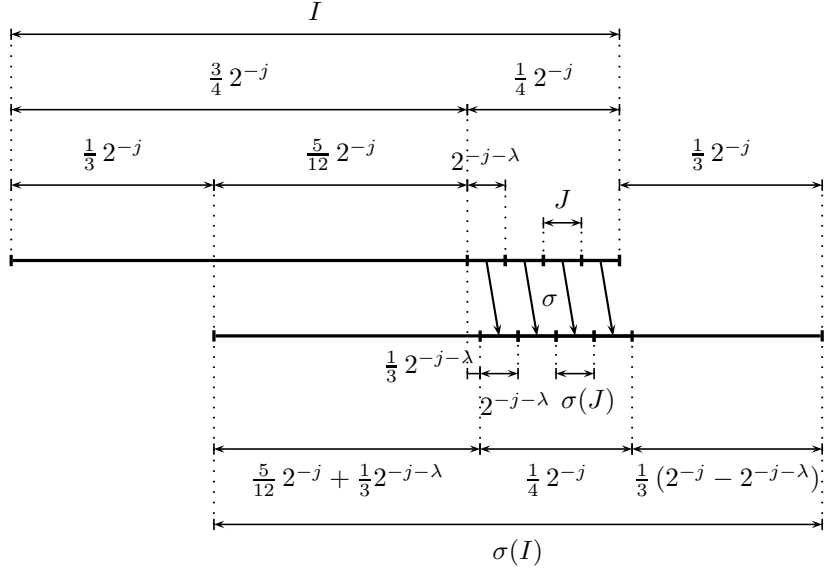
\begin{figure}[bt]
    \begin{center}
      \begin{pspicture}(0,3.5)(12,-4)
        \newcommand{\interval}[1]{%
          \psline[linewidth=1.2pt,tbarsize=0.15]{|*-|*}(#1,0)}

        \interval{8}
        \multirput(6,0)(0.5,0){4}{%
          \interval{0.5}}
        \rput(2.66666667,-1){\interval{8}}
        \multirput(6.16666667,-1)(0.5,0){4}{%
          \interval{0.5}}
        \pnode(6.25,0){top-1}
        \pnode(6.75,0){top-2}
        \pnode(7.25,0){top-3}
        \pnode(7.75,0){top-4}
        \pnode(6.41666667,-1){bottom-1}
        \pnode(6.91666667,-1){bottom-2}
        \pnode(7.41666667,-1){bottom-3}
        \pnode(7.91666667,-1){bottom-4}
        \ncline{->}{top-1}{bottom-1}
        \ncline{->}{top-2}{bottom-2}
        \ncline{->}{top-3}{bottom-3}
        \ncline{->}{top-4}{bottom-4}
        \pcline[linewidth=.6pt,offset=1cm,arrowsize=.1,tbarsize=.15]{|<*->|*}(0,0)(2.66666667,0)
        \naput{$\frac{1}{3}\, 2^{-j}$}
        \pcline[linewidth=.6pt,offset=1cm,arrowsize=.1,tbarsize=.15]{|<*->|*}(2.66666667,0)(6,0)
        \naput{$\frac{5}{12}\, 2^{-j}$}
        \pcline[linewidth=.6pt,offset=1cm,arrowsize=.1,tbarsize=.15]{|<*->|*}(6,0)(6.5,0)
        \naput{$2^{-j-\lambda}$}
        \pcline[linewidth=.6pt,offset=2cm,arrowsize=.1,tbarsize=.15]{|<*->|*}(0,0)(6,0)
        \naput{$\frac{3}{4}\, 2^{-j}$}
        \pcline[linewidth=.6pt,offset=2cm,arrowsize=.1,tbarsize=.15]{|<*->|*}(6,0)(8,0)
        \naput{$\frac{1}{4}\, 2^{-j}$}
        \pcline[linewidth=.6pt,offset=1cm,arrowsize=.1,tbarsize=.15]{|<*->|*}(8,0)(10.66666667,0)
        \naput{$\frac{1}{3}\, 2^{-j}$}
        \rput(0,0.5){%
        \pcline[linewidth=.6pt,offset=-1cm,arrowsize=.1,tbarsize=.15]{|*-|*}(6,-1)(6.16666667,-1)
        \uput{0}[180](6.1,-1.9){$\frac{1}{3}\, 2^{-j-\lambda}$}
        \pcline[linewidth=.6pt,offset=-1cm,arrowsize=.1,tbarsize=.15]{|<*->|*}(6.16666667,-1)(6.66666667,-1)
        \uput{0}[0](6.16666667,-2.4){$2^{-j-\lambda}$}}
        \pcline[linewidth=.6pt,offset=-1.5cm,arrowsize=.1,tbarsize=.15]{|<*->|*}(2.66666667,-1)(6.16666667,-1)
        \nbput{$\frac{5}{12}\, 2^{-j} + \frac{1}{3}2^{-j-\lambda}$}
        \pcline[linewidth=.6pt,offset=-1.5cm,arrowsize=.1,tbarsize=.15]{|<*->|*}(6.16666667,-1)(8.16666667,-1)
        \nbput{$\frac{1}{4}\, 2^{-j}$}
        \pcline[linewidth=.6pt,offset=-1.5cm,arrowsize=.1,tbarsize=.15]{|<*->|*}(8.16666667,-1)(10.66666667,-1)
        \nbput{$\frac{1}{3}\, (2^{-j} - 2^{-j-\lambda})$}
        \psline[linestyle=dotted](0,0)(0,2)
        \psline[linestyle=dotted](2.66666667,-2.5)(2.66666667,1)
        \psline[linestyle=dotted](6,-1.5)(6,2)
        \psline[linestyle=dotted](6.16666667,-2.5)(6.16666667,-1)
        \psline[linestyle=dotted](6.5,0)(6.5,1)
        \psline[linestyle=dotted](6.66666667,-1.5)(6.66666667,-1)
        \psline[linestyle=dotted](8,0)(8,2)
        \psline[linestyle=dotted](8.16666667,-2.5)(8.16666667,-1)
        \psline[linestyle=dotted](10.66666667,-2.5)(10.66666667,1)
        \uput{0}[90](7.08333333,-0.6){$\sigma$}
        \pcline[linewidth=.6pt,offset=3,arrowsize=.1,tbarsize=.15]{|<*->|*}(0,0)(8,0)
        \naput{$I$}
        \pcline[linewidth=.6pt,offset=-3.5,arrowsize=.1,tbarsize=.15]{|<*->|*}(2.66666667,0)(10.66666667,0)
        \nbput{$\sigma(I)$}
        \pcline[linewidth=.6pt,offset=.5,arrowsize=.1,tbarsize=.15]{|<*->|*}(7,0)(7.5,0)
        \naput{$J$}
        \pcline[linewidth=.6pt,offset=-1.5,arrowsize=.1,tbarsize=.15]{|<*->|*}(7.16666667,0)(7.66666667,0)
        \nbput{\quad $\sigma(J)$}
        \psline[linestyle=dotted](0,2)(0,3)
        \psline[linestyle=dotted](8,2)(8,3)
        \psline[linestyle=dotted](7,0)(7,.5)
        \psline[linestyle=dotted](7.5,0)(7.5,.5)
        \psline[linestyle=dotted](7.16666667,-1)(7.16666667,-1.5)
        \psline[linestyle=dotted](7.66666667,-1)(7.66666667,-1.5)
        \psline[linestyle=dotted](2.66666667,-2.5)(2.66666667,-3.5)
        \psline[linestyle=dotted](10.66666667,-2.5)(10.66666667,-3.5)
      \end{pspicture}
    \end{center}
    \caption{The one--third--shift map $\sigma$ acting on
      $I \in \dint$, $|I|=2^{-j}$ and $J \in \dint$,
      $|J|=2^{-j-\lambda}$, where $J \subset I$ and
      $\tau_m(J) \isect I = \emptyset$.
      In this picture $\lambda$ is even.}
    \label{fig:basis_exchange}
  \end{figure}
  Now define for every $0 \leq i \leq L(m)$ the following collections
  of dyadic intervals
  \begin{equation}\label{eqn:splitting_sets--shifted-2}
    \begin{aligned}
      \scr A_i^{(0)}
      & = \Union \big\{
        \scr C_0(\scr A_i, I)\, :\, I \in \scr A_i
        \big\},\\
      \scr A_i^{(1)}
      & = \scr A_i \setminus \scr A_i^{(0)}.
    \end{aligned}
  \end{equation}

  Finally, for all $0 \leq i \leq L(m)$ and $\delta \in \{0,1\}$ we
  split $\scr A_i^{(\delta)}$ into two disjoint collections
  \begin{equation}\label{eqn:splitting_sets--shifted-3}
    \scr B_i^{(\delta)}
    \quad \text{and} \quad
    \scr B_{i+L(m)+1}^{(\delta)},
  \end{equation}
  such that
  \begin{equation}\label{eqn:splitting_sets--shifted-4}
    \scr B_i^{(\delta)} \isect \tau_m\big(\scr B_i^{(\delta)}\big)
    = \emptyset,
  \end{equation}
  for all $0 \leq i \leq K(m)$ and $\delta \in \{0,1\}$, where we set
  $K(m) = 2\cdot L(m) + 1$.
  Considering~\eqref{eqn:dyadic_shift_width} and $L(m) = \lambda - 1$
  we find that $K(m) \leq 7 + 2 \cdot \log_2(m)$.
  For this purpose consider the collection
  \begin{equation*}
    \scr E = \big\{ \tau_k(I) \, :\,
    I \in \dint,\, \inf I = 0,\, 0 \leq k \leq m-1 \big\},
  \end{equation*}
  and observe that
  \begin{equation*}
    \dint
    = \Union_{\substack{j \in \bb Z\\\text{$j$ even}}}
    \tau_{j\cdot m} \big( \scr E \big)
    \union
    \Union_{\substack{j \in \bb Z\\\text{$j$ odd}}}
    \tau_{j\cdot m} \big( \scr E \big)
    = \dint_{\text{even}} \union \dint_{\text{odd}}.
  \end{equation*}
  Now define the collections
  \begin{equation}\label{eqn:splitting_sets--shifted-5}
    \begin{aligned}
      \scr B_i^{(\delta)}
      & = \scr A_i^{(\delta)} \isect \dint_{\text{even}},\\
      \scr B_{i+ L(m) + 1}^{(\delta)}
      & = \scr A_i^{(\delta)} \isect \dint_{\text{odd}},
    \end{aligned}
  \end{equation}
  for all $0 \leq i \leq L(m)$ and $\delta \in \{0,1\}$.

  With regard
  to~\eqref{eqn:splitting_sets--shifted-1},
  \eqref{eqn:lambda_support-1} and noting that
  $\tau_m(I) \in \dint_{\text{odd}}$ if and only if
  $I \in \dint_{\text{even}}$, we
  verified~\eqref{eqn:shift-lemma-nested_collection}, finishing this
  proof.
\end{proof}

\begin{rem}\label{rem:strong_shift-lemma}
  Note that we actually proved the slightly stronger result
  \begin{equation}\label{eqn:strong_shift-lemma}
    I \union \tau_m(I) \subset \pred^\lambda(I),
  \end{equation}
  for all $I \in \sigma^{\delta}\big(\scr B_i^{(\delta)}\big)$,
  $0 \leq i \leq K(m)$, $\delta \in \{0,1\}$.
  Conceive the predecessor map $\pred$ with respect to
  $\sigma^\delta \big(\dint\big)$.
  To be more precise let $I \in \sigma^\delta(\dint)$.
  Then $\pred(I)$ is the unique interval $J \in \sigma^\delta(\dint)$
  such that $J \supset I$, and
  $\pred^\lambda = \pred\circ \cdots \circ \pred$.
\end{rem}

As the combinatorial Lemma~\vref{lem:shift-lemma} exhibits the link between the
shift map $\tau_m$, the one--third--shift map $\sigma$ and Figiel's
compatibility condition~\eqref{eqn:shift-lemma-nested_collection},
the subsequent Theorem~\vref{thm:shift-restricted-tm} will translate the
combinatorial results into analytical results, exhibiting the link between the
shift operator $T_m$, the one--third--shift operator $S$ and martingale
transform operators.

In the following context understand that $1 < p < \infty$, $X$ is a Banach space
with the $\umd$--property and $m \in \bb Z$, $m \neq 0$.
Now we define the projections
$P_i^{(\delta)}\, :\, L_X^p \longrightarrow L_X^p$,
associated with the collections $\scr B_i^{(\delta)}$ in Lemma~\vref{lem:shift-lemma}

\begin{equation}\label{eqn:projection-1}
  P_i^{(\delta)} u
  = \sum_{I \in \scr B_i^{(\delta)}}
    \langle u, h_I \rangle\, h_I\, |I|^{-1},
\end{equation}
for all $0 \leq i \leq K(m)$, $\delta \in \{0,1\}$ and $u \in L_X^p$.
The Banach space $X$ having the $\umd$--property implies uniform bounds on the
projections $P_i^{(\delta)}$.
Note the identity
\begin{equation}\label{eqn:projection-2}
  u = \sum_{\delta \in \{0,1\}} \sum_{i=0}^{K(m)} P_i^{(\delta)} u
\end{equation}
holds true for all $u \in L_X^p$, since the collections
$\scr B_i^{(\delta)}$, $0 \leq i \leq K(m)$, $\delta \in \{0,1\}$ form
a partition of $\dint$, see Lemma~\ref{lem:shift-lemma}.

Exploiting that the one--third--shift operator $S$ is an isomorphism
on $L_X^p$ (see Theorem~\ref{thm:one-third-trick-isomorphism}), we
will now estimate the shift operator $T_m$ on the range of each
$P_i^{(\delta)}$ in the subsequent theorem.
\begin{thm}\label{thm:shift-restricted-tm}
  Let $1 < p < \infty$ and $X$ be a Banach space with the
  $\umd$--property.
  Then for every $m \in \bb Z$, $0 \leq i \leq K(m)$ and
  $\delta \in \{0,1\}$ the inequality
  \begin{equation}\label{eqn:shift_estimate-restricted}
    \big\| T_m \circ P_i^{(\delta)}\, u \big\|_{L_X^p}
    \leq C \cdot \big\| P_i^{(\delta)}\, u \big\|_{L_X^p},
  \end{equation}
  holds true for all $u \in L_X^p$, where the
  constant $C$ depends only on $\umdconst_p(X)$.
  The projections
  $P_i^{(\delta)}$, $0 \leq i \leq K(m)$, $\delta \in \{0,1\}$ are
  defined according to~\eqref{eqn:projection-1}, and
  $K(m) \leq 7 + 2\cdot \log_2(1 + |m|)$.
\end{thm}

\begin{proof}
  Note that due to symmetry once we
  established~\eqref{eqn:shift_estimate-restricted} for $m \geq 1$,
  the theorem is proved.

  Recalling the properties of the partition
  $\scr B_i^{(\delta)}$, $0 \leq i \leq K(m)$, $\delta \in \{0, 1\}$
  of $\dint$, see Lemma~\vref{lem:shift-lemma}, and we know that the
  collection 
  \begin{equation}\label{proof:eqn:shift-lemma-nested_collection}
    \big\{I,\, \tau_m(I),\, I \union \tau_m(I)\, :\,
    I \in \sigma^{\delta}(\scr B_i^{(\delta)}) \big\}
  \end{equation}
  is nested, for all $0 \leq i \leq K(m)$ and $\delta \in \{0, 1\}$.
  Throughout this proof let $m \in \bb Z$, $0 \leq i \leq K(m)$,
  $\delta \in \{0, 1\}$ and $u \in P_i^{(\delta)}(L_X^p)$ be fixed.
  According to~\eqref{eqn:projection-1} we may assume that $u$ has the
  representation
  \begin{equation*}
    u = \sum_{I \in \scr B_i^{(\delta)}} u_I\, h_I\, |I|^{-1}.
  \end{equation*}

  For every $J \in \sigma^\delta(\dint)$ let
  \begin{equation}\label{proof:eqn:shift_atoms-1}
    A^{(\delta)}(J)
    = J \union \tau_m(J),
  \end{equation}
  and for all $j \in \bb Z$ define the collection
  \begin{equation}\label{proof:eqn:shift_atoms-collectoion-1}
    \scr A_j^{(\delta)} =
    \big\{ A^{(\delta)}(J)\, :\, J \in \sigma^\delta(\dint_j) \big\}.
  \end{equation}
  Then specify the filtration $\{\scr F_j^{(\delta)}\}_j$ by
  \begin{equation}\label{proof:eqn:filtration}
    \scr F_j^{(\delta)}
    = \salg \Big( \Union_{i \leq j} \scr A_i^{(\delta)} \Big),
  \end{equation}
  and observe that due
  to~\eqref{proof:eqn:shift-lemma-nested_collection} every
  $A^{(\delta)}(J)$, $J \in \sigma^\delta(\dint_j)$ is an atom for
  $\scr F_j^{(\delta)}$.
  The one--third--shift operator is given by
  \begin{equation}\label{proof:eqn:one-third-shift-operator}
    S^{\delta} u
    = \sum_{I \in \scr B_i^{(\delta)}}
      u_I\, h_{\sigma^{\delta}(I)}\, |I|^{-1}
    = \sum_{ J \in \sigma^\delta (\scr B_i^{(\delta)}) }
      u_{\sigma^{-\delta}(J)}\, h_J\, |J|^{-1},
  \end{equation}
  see~\eqref{eqn:one-third-shift-operator} for details.
  We recall the notation
  \begin{equation*}
    (u)_j = \sum_{|I|=2^{-j}} u_I\, h_I\, |I|^{-1}
    \quad \text{and} \quad
    \indop (u)_j = \sum_{|I|=2^{-j}} u_I\, \charfun_I\, |I|^{-1},
  \end{equation*}
  and note that
  \begin{equation*}
    \big\| T_m\, S^\delta u \big\|_{L_X^p}
    \approx \int_0^1 \big\|
      \sum_{j \in \bb Z} r_j(t)\, \indop\big(T_m\, S^\delta u\big)_j
    \big\|_{L_X^p}
    \, \mathrm dt,
  \end{equation*}
  see~\eqref{eqn:restriction}, \eqref{eqn:indop}
  and~\eqref{eqn:indop-isometry}.
  Obviously,
  $\indop\big(T_m\, S^\delta u\big)_j
  \leq 2\cdot \cond \big(
  \indop(S^\delta u)_j | \scr F_j^{(\delta)} \big)$,
  hence Kahane's contraction principle and Bourgain's version of
  Stein's martingale inequality yield
  \begin{align*}
    \int_0^1 \big\|
      \sum_{j \in \bb Z} r_j(t)\, \indop\big(T_m\, S^\delta u\big)_j
    \big\|_{L_X^p}
    \, \mathrm dt
    & \leq \int_0^1 \big\|
      \sum_{j \in \bb Z} r_j(t)\,
      2\cdot \cond \big(
      \indop(S^\delta u)_j | \scr F_j^{(\delta)} \big)
      \big\|_{L_X^p}
      \, \mathrm dt\\
    & \lesssim \int_0^1 \big\|
      \sum_{j \in \bb Z} r_j(t)\, \indop(S^\delta u)_j \big\|_{L_X^p}
      \, \mathrm dt\\
    & \approx \big\| S^\delta u \big\|_{L_X^p}.
  \end{align*}
  Combining the latter two estimates with
  Theorem~\vref{thm:one-third-trick-isomorphism} proves
  \begin{equation}\label{proof:eqn:shift_estimate-1}
    \big\| T_m\, S^\delta u \big\|_{L_X^p}
    \lesssim \big\| u \big\|_{L_X^p}.
  \end{equation}
  According to~\eqref{eqn:one-third-shift_and_tm_commute}
  the shift operator $T_m$ and the one--third--shift operator $S$
  commute, so we have the identity
  \begin{equation*}
    T_m u
    = \big( S^{-\delta} \circ T_m \circ S^\delta \big)(u),
  \end{equation*}
  and we obtain by an application of
  Theorem~\vref{thm:one-third-trick-isomorphism}
  \begin{equation}\label{proof:eqn:shift_estimate-2}
    \big\| T_m u \big\|_{L_X^p}
    \lesssim \big\|
      \big( T_m \circ S^\delta \big)(u)
    \big\|_{L_X^p}.
  \end{equation}
  We conclude the proof by joining~\eqref{proof:eqn:shift_estimate-2}
  and~\eqref{proof:eqn:shift_estimate-1}.
\end{proof}

\begin{rem}\label{rem:figiel-stein}
  By slightly adjusting the construction of $\scr B_i^{(\delta)}$ we
  could replace Bourgain's version of Stein's martingale inequality by
  the martingale transforms
  in~\cite[Proposition~2, Step~0]{figiel:1988} in order to
  obtain~\eqref{proof:eqn:shift_estimate-1}.
  To this end we will basically have to replace
  $\lambda$ by $\lambda+1$ and redefine $\scr C_0$ and $\scr C_1$ as
  follows
  \begin{equation*}
    \begin{aligned}
      \scr C_0(I, \scr A_i)
      & = \big\{
        J \in \scr A_i\, :\,
        |J| = 2^{-\lambda}\, |I|,\
        J \subset I_0 \text{ and } \tau_m(J) \subset I_0
      \big\}\\
      & \qquad \union
      \big\{
        J \in \scr A_i\, :\,
        |J| = 2^{-\lambda}\, |I|,\
        J \subset I_1 \text{ and } \tau_m(J) \subset I_1
      \big\},\\
      \scr C_1(I, \scr A_i)
      & = \big\{
        J \in \scr A_i\, :\,
        |J| = 2^{-\lambda}\, |I|,\
        J \subset I_0 \text{ and } \tau_m(J) \isect I_0 = \emptyset
      \big\}\\
      & \qquad \union \big\{
        J \in \scr A_i\, :\,
        |J| = 2^{-\lambda}\, |I|,\
        J \subset I_1 \text{ and } \tau_m(J) \isect I_1 = \emptyset
      \big\},
    \end{aligned}
  \end{equation*}
  confer~\eqref{eqn:splitting_sets} and~\eqref{eqn:lambda_support-1}.
  This results in the collection
  \begin{equation}\label{eqn:shift-lemma-nested_collection*1}
    \big\{ J_0,\, \tau_m(J)_0,\, J_1,\, \tau_m(J)_1,\,
    J \union \tau_m(J)\, :\,
    J \in \sigma^{\delta}(\scr B_i^{(\delta)}) \big\}
  \end{equation}
  being nested for all $0 \leq i \leq K(m)$ and
  $\delta \in \{0, 1\}$.
  With this modifications let us define
  \begin{equation*}
    d_{J,1}^{(\delta)}
    = \frac{1}{2}\big( h_J + h_{\tau_m(J)} \big)
    \quad \text{and}\quad
    d_{J,2}^{(\delta)}
    = \frac{1}{2}\big( h_J - h_{\tau_m(J)} \big),
  \end{equation*}
  for all $J \in \sigma^\delta\big( \scr B_i^{(\delta)} \big)$.
  Since~\eqref{eqn:shift-lemma-nested_collection*1} is nested,
  $\big\{ d_{J,1}^{(\delta)}, d_{J,2}^{(\delta)}\, :\,
  J \in \sigma\big( \scr B_i^{(\delta)} \big) \big\}$
  forms a martingale difference sequence.
  Observe $h_J = d_{J,1}^{(\delta)} + d_{J,2}^{(\delta)}$
  and $h_{\tau_m(J)} = d_{J,1}^{(\delta)} - d_{J,2}^{(\delta)}$,
  hence we may swap $h_J$ and $h_{\tau_m(J)}$ without using Bourgain's
  version of Stein's martingale inequality.
\end{rem}
\bigskip
\section{A Martingale Decomposition for $U_m$}\label{s:shift-um}

In this section we will decompose the Haar system
into $24 + 6\cdot \log_2(|m|)$ subcollections, so that on each fixed
subcollection the rearrangement
operator $U_m$ is either a martingale transform operator itself or the
sum of two martingale transform operators.
To be more precise, the total amount of subcollections on which we
will estimate parts of $U_m$ that act as martingale transform
operators will be $40 + 10\cdot \log_2(|m|)$.
This gives immediately the estimate~\cite{figiel:1988}
\begin{equation*}
  \|U_m\, :\, L_X^p \rightarrow L_X^p\|
  \leq C\cdot \big( \log_2 (2 + |m|) \big)^\beta,
\end{equation*}
for some $0 < \beta < 1$.

The operator $T_m$ is easier to analyze than $U_m$.
This is mainly due to the observation that
$\{ T_m h_I \}_{I \in \scr A}$ is a martingale difference sequence for
any choice of $\scr A \subset \dint$, whereas whether
$\{ U_m h_I \}_{I \in \scr B}$ forms a martingale difference sequence
strongly depends on the choice of $\scr B \subset \dint$.
Making use of the one--third--shift operators introduced in
Section~\ref{s:one-third-trick}, we will decompose the operator $U_m$
into the five parts
\begin{equation*}
  U_m = U_m\circ P^{(0)} + \sum_{\varepsilon\in \{0,1\}}
    \big(A_m^{(\varepsilon)} + B_m^{(\varepsilon)}\big)\circ P^{(1,\varepsilon)}
\end{equation*}
each of which behaves like $T_m$.
Some parts of this decomposition will be well localized, whereas others
are widespread, see Figures~\ref{fig:support_functions},
\ref{fig:mds_decomposition0} and~\ref{fig:mds_decomposition1}.

\bigskip

In~\eqref{eqn:shift-taum} we defined the shift map $\tau_m$ for every
$m \in \bb Z$ by
\begin{equation*}
  \tau_m(I) = I + m\, |I|,
\end{equation*}
for all $I \in \dint \union \sigma(\dint)$.
Now we introduce the shift operator $U_m$ by setting
\begin{equation}\label{eqn:shift-um}
  U_m h_I = \charfun_{\tau_m(I)} - \charfun_I,
\end{equation}
for all $I \in \dint \union \sigma(\dint)$.
Essentially the same method we used to bound $T_m$ for functions
supported on the collections $\scr B_i^{(0)}$, $0 \leq i \leq K(m)$
qualifies for estimating $U_m$.
This is primarily due to the fact that
$\big\{ U_m h_I\, :\, I \in \scr B_i^{(0)} \big\}$
forms a martingale difference sequence, which is ensured by
Lemma~\ref{lem:shift-lemma}.
The main obstacle is to estimate $U_m$ on $\scr B_i^{(1)}$, since
$\big\{ U_m h_I\, :\, I \in \scr B_i^{(1)} \big\}$ is \emph{not} a
martingale difference sequence.
The remedy to this problem is the martingale difference sequence
decomposition of $U_m$ into
\begin{equation*}
  U_m h_I = a_I^{(\varepsilon)}
    + b_I^{(\varepsilon)} - b_{\tau_m(I)}^{(\varepsilon)},
  \qquad I \in \scr B_i^{(1,\varepsilon)}
\end{equation*}
where
\begin{align*}
  \scr B_i^{(1,0)}
  &= \big\{I \in \scr B_i^{(1)}\, :\,
    \inf \tau_m(I) \neq \inf \pred^\lambda(\tau_m(I))
  \big\},\\
  \scr B_i^{(1,1)}
  &= \big\{I \in \scr B_i^{(1)}\, :\,
    \inf \tau_m(I) = \inf \pred^\lambda(\tau_m(I))
  \big\}.
\end{align*}
Recall that given $\delta \in \{0,1\}$ and an interval
$I \in \sigma^\delta(\dint)$, the interval $\pred(I)$ is the unique
$J \in \sigma^\delta(\dint)$ such that $J \supset I$, and
$\pred^\lambda = \pred\circ\cdots\circ\pred$.
The collections
$\big\{ a_I^{(\varepsilon)}\, :\, I \in \scr B_i^{(1,\varepsilon)} \big\}$
and
$\big\{ b_I^{(\varepsilon)}, b_{\tau_m(I)}^{(\varepsilon)}\, :\, I \in \scr B_i^{(1,\varepsilon)} \big\}$
are martingale difference sequences, each, see
Theorem~\ref{thm:mds_decomposition}.
This is what enables us to treat $U_m$ like $T_m$, which is elaborated
in Theorem~\ref{thm:shift-restricted-um}.

\bigskip

First, we define
$\alpha_0, \alpha_1\, :\, \dint \longrightarrow \sigma(\dint)$,
\begin{align}
  \alpha_0(I) &= \sigma_0(I),
  \label{eqn:associate-left}\\
  \alpha_1(I) & = \sigma_1(I),
  \label{eqn:associate-right}
\end{align}
where $\sigma_0$, and $\sigma_1$ are given
by~\eqref{eqn:unilateral-map-0} and~\eqref{eqn:unilateral-map-1} in
Subsection~\ref{ss:unilateral}.
Secondly, define the maps $\beta_0, \beta_1$ and $\beta$ by
\begin{align}
  \beta_0 (I) &= \alpha_0(I)\setminus I,
  \label{eqn:beta-left}\\
  \beta_1(I) &= \alpha_1(I)\isect I,
  \label{eqn:beta-right}\\
  \beta(I) &= \beta_0(I) \union \beta_1(I).
  \label{eqn:beta-left+right}
\end{align}
Finally, $\gamma_0, \gamma_1$ and $\gamma$ are given by
\begin{align}
  \gamma_0(I) &= \tau_{-1}(I),
  \label{eqn:gamma-left}\\
  \gamma_1(I) &= I,
  \label{eqn:gamma-right}\\
  \gamma(I) &= \gamma_0(I) \union \gamma_1(I).
\label{eqn:gamma-left+right}
\end{align}
The functions $\alpha_0$, $\alpha_1$, $\beta_0$ and $\beta_1$ are
visualized in Figure~\vref{fig:support_functions}.
\begin{figure}[bt]
  \begin{center}
    \begin{pspicture}(0,-.5)(10,4)
      \newcommand{\interval}[1]{%
        \psline[linewidth=1.2pt,tbarsize=0.15]{|*-|*}(#1,0)}

      \multirput(5.5,0){2}{%
        \interval{1.5}
        \rput(2.25,0){\interval{1.5}}
          \rput(0,1){%
            \interval{2.25}
            \rput(2.25,0){\interval{2.25}}}
          \rput(1.5,2){\interval{2.25}}
        \psline[linestyle=dotted](0,0)(0,3)
        \psline[linestyle=dotted](1.5,0)(1.5,3)
        \psline[linestyle=dotted](2.25,0)(2.25,3)
        \psline[linestyle=dotted](3.75,0)(3.75,3)
        \psline[linestyle=dotted](4.5,1)(4.5,3)

        \pcline[linewidth=.6pt,offset=3,arrowsize=.1,tbarsize=0.15]
        {|<*->|*}(0,0)(1.5,0)
        \naput{$\frac{2}{3}\, |I|$}
        \pcline[linewidth=.6pt,offset=3,arrowsize=.1,tbarsize=0.15]
        {|<*->|*}(1.5,0)(2.25,0)
        \naput{$\frac{1}{3}\, |I|$}
        \pcline[linewidth=.6pt,offset=3,arrowsize=.1,tbarsize=0.15]
        {|<*->|*}(2.25,0)(3.75,0)
        \naput{$\frac{2}{3}\, |I|$}
        \pcline[linewidth=.6pt,offset=3,arrowsize=.1,tbarsize=0.15]
        {|<*->|*}(3.75,0)(4.5,0)
        \naput{$\frac{1}{3}\, |I|$}
      }
      \uput[90](.75,0){$\beta_0(I)$}
      \uput[90](3,0){$\beta_1(I)$}
      \rput(5.5,0){
        \uput[90](.75,0){$\beta_0(\tau_m(I))$}
        \uput[90](3,0){$\beta_1(\tau_m(I))$}}
      \uput[90](1.125,1){$\alpha_0(I)$}
      \uput[90](3.375,1){$\alpha_1(I)$}
      \rput(5.5,0){%
        \uput[90](1.125,1){$\alpha_0(\tau_m(I))$}
        \uput[90](3.375,1){$\alpha_1(\tau_m(I))$}}
      \uput[90](2.625,2){$I$}
      \rput(5.5,0){\uput[90](2.625,2){$\tau_m(I)$}}
    \end{pspicture}
  \end{center}
  \caption{The support functions
    $\alpha_0$, $\alpha_1$, $\beta_0$, $\beta_1$
    for $I$ and $\tau_m(I)$.}
  \label{fig:support_functions}
\end{figure}
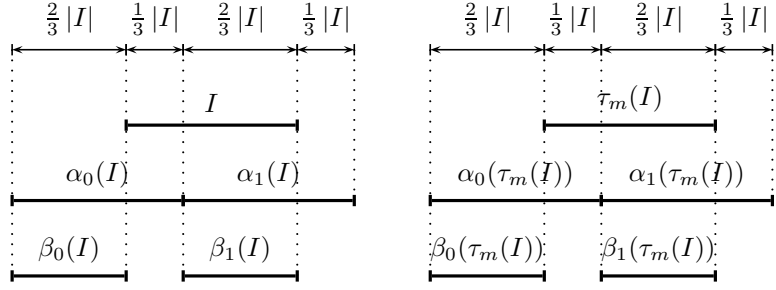

With $m \in  \bb Z$, $m \geq 1$ fixed, we introduce the functions
\begin{align}
  a_I^{(0)} &= \charfun_{\alpha_0(\tau_m(I))} - \charfun_{\alpha_0(I)},
  & I &\in \dint,
  \label{eqn:mds_decomposition-a0}\\
  b_I^{(0)} &= \charfun_{\beta_0(I)} - \charfun_{\beta_1(I)},
  & I &\in \dint,
  \label{eqn:mds_decomposition-b0}
\end{align}
and
\begin{align}
  a_I^{(1)} &= \charfun_{\alpha_1(\tau_m(I))} - \charfun_{\alpha_1(I)},
  & I &\in \dint,
  \label{eqn:mds_decomposition-a1}\\
  b_I^{(1)}
  &= \charfun_{I\setminus\beta_1(I)}-\charfun_{I\setminus\beta_0(I)},
  & I &\in \dint.
  \label{eqn:mds_decomposition-b1}
\end{align}
see Figures~\ref{fig:mds_decomposition0}
and~\ref{fig:mds_decomposition1}.
\begin{figure}[bt]
  \begin{center}
    \begin{pspicture}(0,-1)(10,7)
      \newcommand{\interval}[1]{%
        \psline[linewidth=1.2pt,tbarsize=0.15]{|*-|*}(#1,0)}

      \multirput(5.5,0){2}{%
        \interval{2.25}
        \rput(2.25,0){\interval{2.25}}
        \rput(0,.5){%
          \rput(0,2){%
            \interval{2.25}
            \rput(2.25,0){\interval{2.25}}}
          \rput(1.5,4){\interval{2.25}}}
        \psline[linestyle=dotted](0,0)(0,6)
        \psline[linestyle=dotted](1.5,0)(1.5,6)
        \psline[linestyle=dotted](2.25,0)(2.25,6)
        \psline[linestyle=dotted](3.75,0)(3.75,6)
        \psline[linestyle=dotted](4.5,0)(4.5,6)
      
        \pcline[linewidth=.6pt,offset=6,arrowsize=.1,tbarsize=0.15]
        {|<*->|*}(0,0)(1.5,0)
        \naput{$\frac{2}{3}\, |I|$}
        \pcline[linewidth=.6pt,offset=6,arrowsize=.1,tbarsize=0.15]
        {|<*->|*}(1.5,0)(2.25,0)
        \naput{$\frac{1}{3}\, |I|$}
        \pcline[linewidth=.6pt,offset=6,arrowsize=.1,tbarsize=0.15]
        {|<*->|*}(2.25,0)(3.75,0)
        \naput{$\frac{2}{3}\, |I|$}
        \pcline[linewidth=.6pt,offset=6,arrowsize=.1,tbarsize=0.15]
        {|<*->|*}(3.75,0)(4.5,0)
        \naput{$\frac{1}{3}\, |I|$}}
      \psline(0,.5)(1.5,.5)
      \uput[90](.75,.5){$+1$}
      \uput[-90](.75,0){$\beta_0(I)$}
      \psline(2.25,-.5)(3.75,-.5)
      \uput[-90](3,-.5){$-1$}
      \uput[90](3,0){$\beta_1(I)$}
      \psline[linestyle=dotted](2.25,0)(2.25,-0.5)
      \psline[linestyle=dotted](3.75,0)(3.75,-0.5)
      \uput*[180](5,1){$b_I^{(0)}$\ }
      \rput(5.5,0){
        \psline(0,-.5)(1.5,-.5)
        \uput[-90](.75,-.5){$-1$}
        \uput[90](.75,0){$\beta_0(\tau_m(I))$}
        \psline(2.25,.5)(3.75,.5)
        \uput[90](3,.5){$+1$}
        \uput[-90](3,0){$\beta_1(\tau_m(I))$}
        \psline[linestyle=dotted](0,0)(0,-0.5)
        \psline[linestyle=dotted](1.5,0)(1.5,-0.5)
        \uput*[0](-.5,1){$-b_{\tau_m(I)}^{(0)}$}}
      \rput(0,.5){%
      \psline(0,1.5)(2.25,1.5)
      \uput[-90](1.125,1.5){$-1$}
      \uput[90](1.125,2){$\alpha_0(I)$}
      \uput[90](3.375,2){$\alpha_1(I)$}
      \psline(5.5,2.5)(7.75,2.5)
      \rput(5.5,0){%
        \uput[90](1.125,2.5){$+1$}
        \uput[-90](1.125,2){$\alpha_0(\tau_m(I))$}
        \uput[-90](3.375,2){$\alpha_1(\tau_m(I))$}}
      \uput{0}[90](5,2){$a_I^{(0)}$}
      \psline(1.5,3.5)(3.75,3.5)
      \psline(7,4.5)(9.25,4.5)
      \uput[90](2.625,4){$I$}
      \rput(5.5,0){\uput[270](2.625,4){$\tau_m(I)$}}
      \uput[270](2.625,3.5){$-1$}
      \rput(5.5,0){\uput[90](2.625,4.5){$+1$}}
      \uput{0}[90](5,4){$U_m h_I$}}
    \end{pspicture}
  \end{center}
  \caption{Martingale decomposition of $U_m$ to the left.}
  \label{fig:mds_decomposition0}
\end{figure}
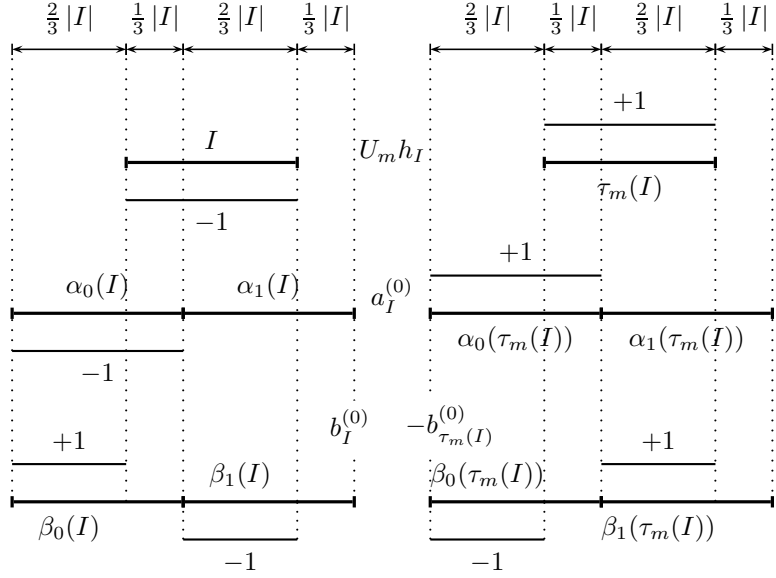
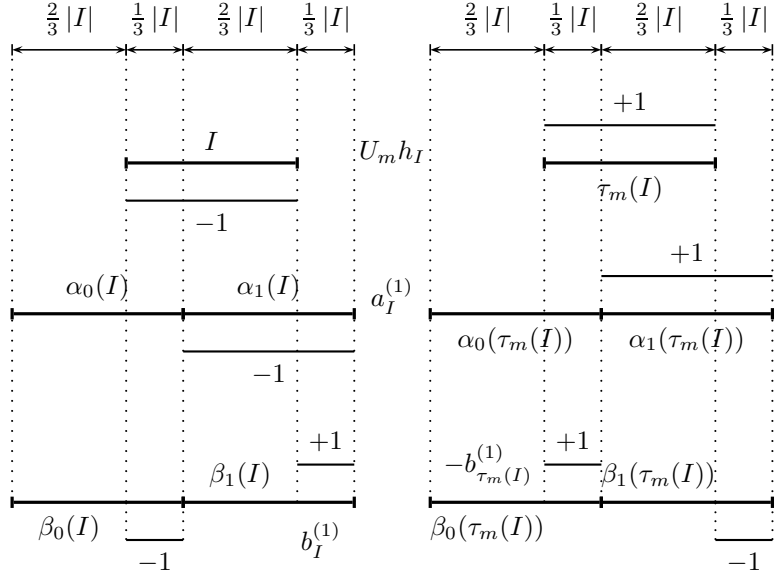
\begin{figure}[bt]
  \begin{center}
    \begin{pspicture}(0,-1)(10,7)
      \newcommand{\interval}[1]{%
        \psline[linewidth=1.2pt,tbarsize=0.15]{|*-|*}(#1,0)}

      \multirput(5.5,0){2}{%
        \interval{2.25}
        \rput(2.25,0){\interval{2.25}}
        \rput(0,.5){%
          \rput(0,2){%
            \interval{2.25}
            \rput(2.25,0){\interval{2.25}}}
          \rput(1.5,4){\interval{2.25}}}
        \psline[linestyle=dotted](0,0)(0,6)
        \psline[linestyle=dotted](1.5,0)(1.5,6)
        \psline[linestyle=dotted](2.25,0)(2.25,6)
        \psline[linestyle=dotted](3.75,0)(3.75,6)
        \psline[linestyle=dotted](4.5,0)(4.5,6)
      
        \pcline[linewidth=.6pt,offset=6,arrowsize=.1,tbarsize=0.15]
        {|<*->|*}(0,0)(1.5,0)
        \naput{$\frac{2}{3}\, |I|$}
        \pcline[linewidth=.6pt,offset=6,arrowsize=.1,tbarsize=0.15]
        {|<*->|*}(1.5,0)(2.25,0)
        \naput{$\frac{1}{3}\, |I|$}
        \pcline[linewidth=.6pt,offset=6,arrowsize=.1,tbarsize=0.15]
        {|<*->|*}(2.25,0)(3.75,0)
        \naput{$\frac{2}{3}\, |I|$}
        \pcline[linewidth=.6pt,offset=6,arrowsize=.1,tbarsize=0.15]
        {|<*->|*}(3.75,0)(4.5,0)
        \naput{$\frac{1}{3}\, |I|$}}
      \psline(1.5,-.5)(2.25,-.5)
      \uput[-90](1.875,-.5){$-1$}
      \uput[-90](.75,0){$\beta_0(I)$}
      \psline(3.75,.5)(4.5,.5)
      \uput[90](4.125,.5){$+1$}
      \uput[90](3,0){$\beta_1(I)$}
      \psline[linestyle=dotted](1.5,0)(1.5,-0.5)
      \psline[linestyle=dotted](2.25,0)(2.25,-0.5)
      \uput*[180](4.5,-.5){$b_I^{(1)}$}
      \rput(5.5,0){
        \psline(1.5,.5)(2.25,.5)
        \uput[90](1.875,.5){$+1$}
        \uput[-90](.75,0){$\beta_0(\tau_m(I))$}
        \psline(3.75,-.5)(4.5,-.5)
        \uput[-90](4.125,-.5){$-1$}
        \uput[90](3,0){$\beta_1(\tau_m(I))$}
        \psline[linestyle=dotted](3.75,0)(3.75,-0.5)
        \psline[linestyle=dotted](4.5,0)(4.5,-0.5)
        \uput*[0](0,.5){$-b_{\tau_m(I)}^{(1)}$}
      }
      \rput(0,.5){%
      \psline(2.25,1.5)(4.5,1.5)
      \uput[-90](3.375,1.5){$-1$}
      \uput[90](1.125,2){$\alpha_0(I)$}
      \uput[90](3.375,2){$\alpha_1(I)$}
      \psline(7.75,2.5)(10,2.5)
      \rput(5.5,0){%
        \uput[90](3.375,2.5){$+1$}
        \uput[-90](1.125,2){$\alpha_0(\tau_m(I))$}
        \uput[-90](3.375,2){$\alpha_1(\tau_m(I))$}}
      \uput{0}[90](5,2){$a_I^{(1)}$}
      \psline(1.5,3.5)(3.75,3.5)
      \psline(7,4.5)(9.25,4.5)
      \uput[90](2.625,4){$I$}
      \rput(5.5,0){\uput[270](2.625,4){$\tau_m(I)$}}
      \uput[270](2.625,3.5){$-1$}
      \rput(5.5,0){\uput[90](2.625,4.5){$+1$}}
      \uput{0}[90](5,4){$U_m h_I$}}
    \end{pspicture}
  \end{center}
  \caption{Martingale decomposition of $U_m$ to the right.}
  \label{fig:mds_decomposition1}
\end{figure}
We define the operators $A_m^{(\varepsilon)}$, $B^{(\varepsilon)}$ and
$B_m^{(\varepsilon)}$ as the linear extension of
\begin{align}
  A_m^{(\varepsilon)}\, h_I
  &= a_I^{(\varepsilon)},
  & I &\in \dint,
  \label{eqn:mds_decomposition-Am}\\
  B^{(\varepsilon)}\, h_I
  &= b_I^{(\varepsilon)},
  & I &\in \dint,
  \label{eqn:mds_decomposition-B}\\
  B_m^{(\varepsilon)}\, h_I
  &= b_I^{(\varepsilon)} - b_{\tau_m(I)}^{(\varepsilon)},
  & I &\in \dint,
  \label{eqn:mds_decomposition-Bm}
\end{align}
for $\varepsilon \in \{0,1\}$.
Note the identities
\begin{equation}\label{eqn:mds_decomposition-1}
  U_m = A_m^{(\varepsilon)} + B_m^{(\varepsilon)}
  = A_m^{(\varepsilon)} + B^{(\varepsilon)} - B^{(\varepsilon)}\circ T_m,
\end{equation}
hold true for $\varepsilon \in \{0,1\}$, see
\eqref{eqn:mds_decomposition-a0}, \eqref{eqn:mds_decomposition-b0},
\eqref{eqn:mds_decomposition-a1}, \eqref{eqn:mds_decomposition-b1} and
Figures~\ref{fig:mds_decomposition0} and~\ref{fig:mds_decomposition1}.

Now we split the collections $\scr B_i^{(1)}$ into
\begin{equation}\label{eqn:split}
  \scr B_i^{(1)} = \scr B_i^{(1,0)} \union   \scr B_i^{(1,1)},
\end{equation}
where
\begin{align}
  \scr B_i^{(1,0)}
  &= \big\{I \in \scr B_i^{(1)}\, :\,
    \inf \tau_m(I) \neq \inf \pred^\lambda(\tau_m(I))
  \big\},
  \label{eqn:split-0}\\
  \scr B_i^{(1,1)}
  &= \big\{I \in \scr B_i^{(1)}\, :\,
    \inf \tau_m(I) = \inf \pred^\lambda(\tau_m(I))
  \big\},
  \label{eqn:split-1}
\end{align}
for all $0 \leq i \leq K(m)$.
The projections
\begin{equation*}
  P_i^{(0)} u
  = \sum_{I \in \scr B_i^{(0)}}
  \langle u, h_I \rangle\, h_I\, |I|^{-1}
\end{equation*}
were defined in~\eqref{eqn:projection-1}, accordingly we set
\begin{equation}\label{eqn:projection-1*}
  P_i^{(1,\varepsilon)} u
  = \sum_{I \in \scr B_i^{(1,\varepsilon)}}
  \langle u, h_I \rangle\, h_I\, |I|^{-1},
\end{equation}
for all $0 \leq i \leq K(m)$ and $\varepsilon \in \{0,1\}$.
The collection $\scr B_i^{(0)}$ is specified in
Lemma~\ref{lem:shift-lemma}, and $\scr B_i^{(1,\varepsilon)}$ is
defined in~\eqref{eqn:split-0} and~\eqref{eqn:split-1}.
Finally, if we define
\begin{equation}\label{eqn:projection-3}
  \begin{aligned}
    P^{(0)} &= \sum_{i=0}^{K(m)} P_i^{(0)},\\
    P^{(1,\varepsilon)} &= \sum_{i=0}^{K(m)} P_i^{(1,\varepsilon)},
  \end{aligned}
\end{equation}
for all $\varepsilon \in \{0,1\}$, then certainly
\begin{equation}\label{eqn:projection-4}
  u = P^{(0)}\, u + P^{(1,0)}\, u + P^{(1,1)}\, u
\end{equation}
for all $u \in L_X^p$.
Note that $P_i^{(1)} = P_i^{(1,0)} + P_i^{(1,1)}$, where
$P_i^{(1)}$ was defined in~\eqref{eqn:projection-1}.

In the following theorem the operator $U_m$ is decomposed into five
parts, each of which is forming a martingale difference sequence.
\begin{thm}\label{thm:mds_decomposition}
  Let $m \in \bb Z$, $m \geq 1$ and $0 \leq i \leq K(m)$.
  Then the identity
  \begin{equation}\label{eqn:mds_decomposition-um-3}
    U_m\, u = U_m \circ P^{(0)}\, u
      + \sum_{\varepsilon \in \{0,1\}}
      \big( A_m^{(\varepsilon)} +B_m^{(\varepsilon)}\big)
        \circ P^{(1,\varepsilon)}\, u
  \end{equation}
  holds true for all $u \in L_X^p$.
  For every $0 \leq i \leq K(m)$ and $\varepsilon \in \{0,1\}$, each
  of the following collections is a martingale difference sequence:
  \begin{align}
    \big\{ U_m \circ P_i^{(0)}\, h_I\, & :\, I \in \dint \big\},
    \label{eqn:mds-0}\\
    \big\{
      A_m^{(\varepsilon)} \circ P_i^{(1,\varepsilon)}\, h_I\, & :\,
    I \in \dint \big\},
    \label{eqn:mds-a}\\
    \big\{
      B_m^{(\varepsilon)} \circ P_i^{(1,\varepsilon)}\, h_I\, & :\,
    I \in \dint \big\}.
    \label{eqn:mds-b}
  \end{align}
  We have the estimate $K(m) \leq 7 + 2\cdot \log_2(m)$, where $K(m)$
  is defined in Lemma~\vref{lem:shift-lemma}.
\end{thm}

\begin{rem}
  For reasons of symmetry, a similar result holds true for
  $m \leq -1$, when adjusting the construction of $a_I$, $b_I$ and
  $\scr B_i^{(1,\varepsilon)}$, accordingly.
\end{rem}

\begin{proof}
  Let $m \in \bb Z$, $m \geq 1$ and $0 \leq i \leq K(m)$ be fixed
  throughout the rest of this proof.
  Whenever we apply the predecessor map $\pred$ to an interval 
  $I \in \sigma^\delta\big(\dint\big)$, we understand it with respect
  to $\sigma^\delta \big(\dint\big)$, with $\delta \in \{0,1\}$ fixed.

  Observe, identity~\eqref{eqn:mds_decomposition-um-3} follows
  immediately from~\eqref{eqn:projection-4}
  and~\eqref{eqn:mds_decomposition-1}.

  First, note that Lemma~\ref{lem:shift-lemma} implies that
  \begin{equation*}
    \big\{ I,\, \tau_m(I),\, I \union \tau_m(I)\, :\,
    I \in \scr B_i^{(0)} \big\}
  \end{equation*}
  is a nested collection of sets, hence
  \begin{equation*}
    \big\{ U_m\, h_I\, :\, I \in \scr B_i^{(0)} \big\}
  \end{equation*}
  is a martingale difference sequence.

  Secondly, we will show that
  $\big\{ a_I^{(0)}\, :\, I \in \scr B_i^{(1,0)} \big\}$ forms a
  martingale difference sequence.
  Henceforth, we shall abbreviate $\scr B_i^{(1,0)}$ by $\scr B$.
  Now, fix $I, J \in \scr B$, $|J| < |I|$ such that
  $\supp a_J^{(0)} \isect \supp a_I^{(0)} \neq \emptyset$.
  Note that
  $J \subset \big(\pred^{\lambda}(J)\big)_{11}$, for all
  $J \in \scr B$, where $K_{11}$, $K \in \dint$ denotes the
  unique $M \subset K$, $M \in \dint$, $|M| = |K|/4$ such that
  $\sup M = \sup K$.
  From this and the definition of $\scr B$ it is clear that
  $\supp a_J^{(0)} \subset \alpha_1(\pred^\lambda(J))$
  (see also Remark~\ref{rem:strong_shift-lemma}), hence
  \begin{equation*}
    \emptyset
    \neq \alpha_1(\pred^\lambda(J)) \isect \supp a_I^{(0)}
    = \big( \alpha_1(\pred^\lambda(J)) \isect \alpha_0(I) \big)
    \union \big(
      \alpha_1(\pred^\lambda(J)) \isect \alpha_0(\tau_m(I)) \big).
  \end{equation*}
  Since $|J| < |I|$, $I, J \in \scr B$, we know that
  $|\alpha_1(\pred^\lambda(J))| \leq |I|$, thus
  \begin{equation*}
    \text{either}\quad
    \alpha_1(\pred^\lambda(J)) \subset \alpha_0(I)
    \quad \text{or}\quad
    \alpha_1(\pred^\lambda(J)) \subset \alpha_0(\tau_m(I)),
  \end{equation*}
  which finishes the second part of this proof.

  The proof that
  $\big\{ a_I^{(1)}\, :\, I \in \scr B_i^{(1,1)} \big\}$ forms a
  martingale difference sequence is essentially the same, and we omit
  the details.

  Thirdly, we will show that
  $\big\{ b_I^{(0)}, b_{\tau_m(I)}^{(0)}\, :\, I \in \scr B^{(1,0)} \big\}$
  constitutes a martingale difference sequence.
  Again, we shall abbreviate $\scr B_i^{(1,0)}$ by $\scr B$.
  To this end, we assume there exist
  $I, J \in \scr B \union \tau_m(\scr B)$, $|J| < |I|$ such that
  \begin{equation}\label{proof:assumption}\tag{$\cal A$}
    \beta(J) \isect \beta(I) \neq \emptyset
    \qquad \text{and}\qquad
    \beta(J) \isect \beta(I)^c \neq \emptyset.
  \end{equation}
  Since $\beta(J) \subset \gamma(J)$,
  assumption~\eqref{proof:assumption} is covered by the following four
  cases.
  \begin{enumerate}
  \item $\gamma(J) \isect I \neq \emptyset$ and
    $\gamma(J) \isect I^c \neq \emptyset$,
    \label{enu:case-1}
  \item $\gamma(J) \isect \gamma_0(I) \neq \emptyset$ and
    $\gamma(J) \isect \gamma_0(I)^c \neq \emptyset$,
    \label{enu:case-2}
  \item $\gamma(J) \subset I$ and
    $\inf \beta_1(I) \in \gamma(J)$,
    \label{enu:case-3}
  \item $\gamma(J) \subset \gamma_0(I)$ and
    $\inf \beta_0(I) \in \gamma(J)$.
    \label{enu:case-4}
  \end{enumerate}
  If we assume case~\eqref{enu:case-1}, then $\inf J = \inf I$ or
  $\inf J = \sup I$. Anyhow, we have that
  $\inf J = \inf \pred^\lambda(J)$, so we know
  $J \notin \big( \scr B \union \tau_m(\scr B) \big)$,
  contradicting our assumption.
  Case~\eqref{enu:case-2} is analogous to case~\eqref{enu:case-1}.
  Note that we abbreviated $\scr B_i^{(1,0)}$ by $\scr B$, so
  consider the definition of $\scr B_i^{(1)}$ to see that
  $J \notin \scr B_i^{(1,0)}$, and consider~\eqref{eqn:split-0} to
  determine that also $J \notin \tau_m(\scr B_i^{(1,0)})$.

  Let us now assume case~\eqref{enu:case-3} is true.
  This means that either
  $\inf I + \frac{1}{3} |I| \in \gamma(J)$ or
  $\inf I + \frac{2}{3} |I| \in \gamma(J)$, depending on the sign of
  the one--third--shift for $I$.
  We fix $z \in \{1,2\}$ and assume that
  \begin{equation}\label{proof:case-3-assm-1}
    \inf I + \frac{z}{3} |I| \in \gamma(J).
  \end{equation}
  Due to~\eqref{eqn:split-0} we see that
  $\pred^\lambda(\gamma_0(J)) = \pred^\lambda(J)$, so if
  we set $K = \pred^\lambda(J)$, then
  \begin{equation*}
    \inf I + \frac{z}{3} |I| \in K.
  \end{equation*}
  This corresponds to either one of the following being true
  \begin{equation}\label{proof:case-3-assm-2}
    \inf I + \frac{z}{3} |I| = \inf K + \frac{1}{3} |K|
    \qquad \text{or}\qquad
    \inf I + \frac{z}{3} |I| = \inf K + \frac{2}{3} |K|.
  \end{equation}

  \noindent
  If $J \in \scr B$ we know $J \subset K_{11}$, thus
  \begin{equation}\label{proof:case-3-assm-3}
    \begin{aligned}
      \inf \gamma(J)
      & \geq \inf K + \frac{3}{4}|K| - 2^{-\lambda} |K|\\
      & > \inf K + \frac{2}{3}|K|.
    \end{aligned}
  \end{equation}
  Recall that $K_{11}$ denotes the unique
  $M \subset K$, $M \in \dint$, $|M| = |K|/4$ such that
  $\sup M = \sup K$.
  The last strict inequality holds true since $\lambda \geq 4$ per
  construction of $\scr B$, see~\eqref{eqn:dyadic_shift_width} if
  $|m| \geq 2$ and note the exception for $|m|=1$ beneath.
  Combining~\eqref{proof:case-3-assm-1}
  and~\eqref{proof:case-3-assm-3} yields
  \begin{equation*}
    \inf I + \frac{z}{3} |I|
    > \inf K + \frac{2}{3}|K|,
  \end{equation*}
  which contradicts~\eqref{proof:case-3-assm-2} in both cases.

  \noindent
  If $J \in \scr \tau_m(\scr B)$ we know $J \subset K_{00}$, where
  $K_{00}$ denotes the unique $M \subset K$, $M \in \dint$,
  $|M| = |K|/4$ such that $\inf M = \inf K$.
  So we note
  \begin{equation}\label{proof:case-3-assm-4}
    \begin{aligned}
      \sup \gamma(J)
      & \leq \inf K + \frac{1}{4}|K| + 2^{-\lambda} |K|\\
      & < \inf K + \frac{1}{3}|K|.
    \end{aligned}
  \end{equation}
  The last strict inequality holds true since $\lambda \geq 4$ per
  construction of $\scr B$, see~\eqref{eqn:dyadic_shift_width} if
  $|m| \geq 2$ and note the exception for $|m|=1$ beneath.
  Combining~\eqref{proof:case-3-assm-1}
  and~\eqref{proof:case-3-assm-4} yields
  \begin{equation*}
    \inf I + \frac{z}{3} |I|
    < \inf K + \frac{1}{3}|K|,
  \end{equation*}
  which contradicts~\eqref{proof:case-3-assm-2} in both cases.

  Case~\eqref{enu:case-4} is analogous to case~\eqref{enu:case-3}.

  Altogether we proved that our assumption~\eqref{proof:assumption}
  was false, therefore
  \begin{equation*}
    \beta(J) \subset \beta_0(I)
    \qquad \text{or}\qquad
    \beta(J) \subset \beta_1(I)
  \end{equation*}
  for all $I, J \in \scr B$, $|J| < |I|$
  such that $\beta(J) \isect \beta(I) \neq \emptyset$.
  In other words, the support of $b_J$ is contained in a set where
  $b_I^{(0)}$ is constant, hence
  \begin{equation*}
    \big\{ b_I^{(0)}, b_{\tau_m(I)}^{(0)}\, :\, I \in \scr B_i^{(1,0)} \big\}
  \end{equation*}
  constitutes a martingale difference sequence.

  The proof that
  $\big\{ b_I^{(1)}, b_{\tau_m(I)}^{(1)}\, :\, I \in \scr B^{(1,1)} \big\}$
  constitutes a martingale difference sequence is essentially the same
  argument, so we omit it.
\end{proof}

\begin{rem}\label{rem:mds_decomposition}
  Note in Theorem~\ref{thm:mds_decomposition} we actually proved
  the following stronger result.
  For every $0 \leq i \leq K(m)$ and $\varepsilon \in \{0,1\}$, the
  collection
  \begin{equation*}
    \big\{ b_I^{(\varepsilon)}, b_{\tau_m(I)}^{(\varepsilon)}\, :\,
    I \in \scr B_i^{(1,\varepsilon)} \big\}
  \end{equation*}
  is a martingale difference sequence, which certainly
  implies~\eqref{eqn:mds-b}.
\end{rem}

Consider the splitting of $\dint$ into the sets
$\scr B_i^{(\delta)}$, $0 \leq i \leq K(m)$, $\delta \in \{0,1\}$,
see Lemma~\vref{lem:shift-lemma} for details, which we used in
Theorem~\vref{thm:shift-restricted-tm} to treat the shift operator
$T_m$. Retracing our steps in the proof of
Theorem~\ref{thm:shift-restricted-tm} we find that we could actually
repeat this proof with the operator $T_m$ replaced by any of the
operators
$U_m \circ P^{(0)}$, $A_m^{(\varepsilon)}\circ P^{(1,\varepsilon)}$,
$B_m^{(\varepsilon)}\circ P^{(1,\varepsilon)}$,
$\varepsilon \in \{0,1\}$.
The details are elaborated in
Theorem~\ref{thm:shift-restricted-um} below.
\begin{thm}\label{thm:shift-restricted-um}
  Let $m \in \bb Z$ and $m \geq 1$.
  Then for all $0 \leq i \leq K(m)$ and $\varepsilon \in \{0,1\}$, we
  have the estimates
  \begin{equation}\label{eqn:shift_estimate-restricted-um-1}
    \begin{aligned}
    \big\| U_m \circ P_i^{(0)}\, u \big\|_{L_X^p}
    & \leq C \cdot \big\| P_i^{(0)}\, u \big\|_{L_X^p},\\
    \big\| U_m \circ P_i^{(1,\varepsilon)}\, u \big\|_{L_X^p}
    & \leq C \cdot \big\| P_i^{(1,\varepsilon)}\, u \big\|_{L_X^p},
    \end{aligned}
  \end{equation}
  for all $u \in L_X^p$, where the constant $C$ depends only on
  $\umdconst_p(X)$.
  Furthermore, we have the bound $K(m) \leq 7 + 2\cdot \log_2(m)$.
\end{thm}

\begin{rem}
  For reasons of symmetry, the same result holds true for
  $m \leq -1$, that is besides the appropriate modifications for
  $P_i^{(0)}$ and $P_i^{(1,\varepsilon)}$.
\end{rem}

\begin{proof}
  Let $m \in \bb Z$, $m \geq 1$ and $0 \leq i \leq K(m)$ be fixed
  throughout the rest of the proof.
  
  First, we will estimate $U_m \circ P_i^{(0)}$.
  Due to Theorem~\ref{thm:mds_decomposition} respectively
  Remark~\ref{rem:mds_decomposition} we know that
  $\big\{ U_m\circ P_i^{(0)} h_I\, :\, I \in \dint \big\}$
  forms a martingale difference sequence, which enables us to
  introduce Rademacher functions via the $\umd$--property.
  Hence
  \begin{align*}
    \big\| U_m \circ P_i^{(0)}\, u \big\|_{L_X^p}
    & \approx \int_0^1 \Big\| \sum_{I \in \scr B_i^{(0)}}
      r_I(t)\, \langle u, h_I \rangle\, U_m\, h_I\, |I|^{-1}
    \Big\|_{L_X^p}\, \mathrm dt\\
    & = \int_0^1 \Big\| \sum_{I \in \scr B_i^{(0)}}
      r_I(t)\, \langle u, h_I \rangle\, (\Id + T_m)\, h_I\, |I|^{-1}
    \Big\|_{L_X^p}\, \mathrm dt
  \end{align*}
  for all $u \in L_X^p$.
  This is all we need to repeat the proof of
  Theorem~\ref{thm:shift-restricted-tm} in Section~\ref{s:shift-tm}
  with $T_m$ replaced by $\Id + T_m$.

  Now we turn to the estimate for $U_m \circ P_i^{(1,\varepsilon)}$,
  with $\varepsilon \in \{0,1\}$ fixed throughout the rest of the
  proof.
  Observe that
  \begin{equation*}
    U_m \circ P_i^{(1,\varepsilon)}\, u
    = A_m^{(\varepsilon)} \circ P_i^{(1,\varepsilon)}\, u
    + B_m^{(\varepsilon)} \circ P_i^{(1,\varepsilon)}\, u,
  \end{equation*}
  for all $u \in L_X^p$, see~\eqref{eqn:mds_decomposition-1}.
  Theorem~\vref{thm:mds_decomposition} ensures that both
  \begin{equation*}
    \big\{ A_m^{(\varepsilon)}\circ P_i^{(1,\varepsilon)}\, h_I\, :\,
    I \in \dint \big\}
    \quad \text{and}\quad
    \big\{ B_m^{(\varepsilon)}\circ P_i^{(1,\varepsilon)}\, h_I\, :\,
    I \in \dint \big\}
  \end{equation*}
  form martingale difference sequences, which allows us to introduce
  Rademacher means via the $\umd$--property, hence
  \begin{align*}
    \big\|
      A_m^{(\varepsilon)} \circ P_i^{(1,\varepsilon)}\, u
    \big\|_{L_X^p}
    &\lesssim \int_0^1 \Big\| \sum_{I \in \scr B_i^{(1,\varepsilon)}}
      r_I(t)\, \langle u, h_I \rangle\, a_I^{(\varepsilon)}\, |I|^{-1}
    \Big\|_{L_X^p} \mathrm dt
    \intertext{and}
    \big\|
      B_m^{(\varepsilon)} \circ P_i^{(1,\varepsilon)}\, u
    \big\|_{L_X^p}
    & \lesssim \int_0^1 \Big\| \sum_{I \in \scr B_i^{(1,\varepsilon)}}
      r_I(t)\, \langle u, h_I \rangle\,
      \big(
        b_I^{(\varepsilon)} - b_{\tau_m(I)}^{(\varepsilon)}
      \big)\,
      h_I\, |I|^{-1}
    \Big\|_{L_X^p} \mathrm dt,
  \end{align*}
  for all $u \in L_X^p$.
  Now we can essentially repeat the proof of
  Theorem~\ref{thm:shift-restricted-tm} in Section~\ref{s:shift-tm},
  for $\delta=1$ and with $T_m$ replaced by $A_m^{(\varepsilon)}$ and
  $B_m^{(\varepsilon)}$, respectively. We have to utilize the
  unilateral operators $S_0$ and $S_1$ instead of $S$ as well, see
  Subsection~\vref{ss:unilateral}.
  If we do so, we end up with the estimates
  \begin{align*}
    \big\|
      A_m^{(\varepsilon)} \circ P_i^{(1,\varepsilon)}\, u
    \big\|_{L_X^p}
    &\lesssim \int_0^1 \Big\| \sum_{I \in \scr B_i^{(1,\varepsilon)}}
      r_I(t)\, \langle u, h_I \rangle\, h_{\alpha_\varepsilon(I)}\, |I|^{-1}
    \Big\|_{L_X^p} \mathrm dt
    \intertext{and}
    \big\|
      B_m^{(\varepsilon)} \circ P_i^{(1,\varepsilon)}\, u
    \big\|_{L_X^p}
    & \lesssim \int_0^1 \Big\| \sum_{I \in \scr B_i^{(1,\varepsilon)}}
      r_I(t)\, \langle u, h_I \rangle\, b_I^{(\varepsilon)}\, |I|^{-1}
    \Big\|_{L_X^p} \mathrm dt,
  \end{align*}
  for all $u \in L_X^p$.
  Thus, considering $h_{\alpha_\varepsilon(I)} = S_\varepsilon h_I$ and
  $|b_I^{(\varepsilon)}| \leq |S_0 h_I| + |S_1 h_I|$,
  see~\eqref{eqn:unilateral-map-0}, \eqref{eqn:unilateral-map-1},
  \eqref{eqn:unilateral-operator-0}, \eqref{eqn:unilateral-operator-1}
  and combining our estimates for $A_m^{(\varepsilon)}$ and
  $B_m^{(\varepsilon)}$ with the inequalities for the unilateral
  one-third-shift operators $S_0$ and $S_1$ in
  Theorem~\vref{thm:modified-one-third-trick-isomorphism}
  yields
  \begin{align*}
    \big\| U_m\circ P_i^{(1,\varepsilon)}\, u \big\|_{L_X^p}
    & \lesssim \big\| S_0\circ P_i^{(1,\varepsilon)}\, u \big\|_{L_X^p}
      + \big\| S_1\circ P_i^{(1,\varepsilon)}\, u \big\|_{L_X^p}\\
    & \lesssim \big\| P_i^{(1,\varepsilon)}\, u \big\|_{L_X^p},
  \end{align*}
  for all $u \in L_X^p$, concluding the proof.
\end{proof}

From the results established in Theorem~\vref{thm:shift-restricted-tm} and
Theorem~\vref{thm:shift-restricted-um} one can obtain the estimates stated in
Theorem~\ref{thm:figiel} below, by exploiting the type and cotype inequalities
for $T_m$, and only the cotype inequality for $U_m$.
Inserting Theorem~\vref{thm:shift-restricted-tm} and
Theorem~\vref{thm:shift-restricted-um}
into~\cite[Lemma 1]{figiel:1988} one can
obtain~\cite[Theorem 1]{figiel:1988} stated below for sake of
completeness.
\begin{thm}[\cite{figiel:1988}]\label{thm:figiel}
  Let $1 < p < \infty$, and $X$ be a Banach space with the
  $\umd$--property. For $m \in \bb Z$ let the map
  $\tau_m$ denote the shift map defined by
  \begin{equation*}
    I \mapsto I + m\, |I|.
  \end{equation*}
  Let $T_m$, $U_m$ denote the linear extensions of the maps
  \begin{align*}
    T_m h_I &= h_{\tau_m(I)},
    \intertext{and}
    U_m h_I &= \charfun_{\tau_m(I)} - \charfun_I,
  \end{align*}
  respectively, then
  \begin{align*}
    \|T_m\, :\, L_X^p \rightarrow L_X^p\|
    & \leq C\, \big( \log_2 (2 + |m|) \big)^\alpha,\\
    \|U_m\, :\, L_X^p \rightarrow L_X^p\|
    & \leq C\, \big( \log_2 (2 + |m|) \big)^\beta,
  \end{align*}
  where the constant $C > 0$ depends only $\umdconst_p(X)$ and
  $0 < \alpha, \beta < 1$.
  Moreover, if $L_X^p$ has type $\cal T$ and cotype $\cal C$,
  then one can take
  $\alpha = \frac{1}{\cal T} - \frac{1}{\cal C}$ and
  $\beta = 1 - \frac{1}{\cal C}$.
\end{thm}

\bibliographystyle{alpha}
\bibliography{bibliography}

\begin{thebibliography}{CWW85}

\bibitem[Bou86]{bourgain:1986}
J.~Bourgain.
\newblock {Vector-valued singular integrals and the $H\sp 1$-BMO duality.}
\newblock {Probability theory and harmonic analysis, Pap. Mini-Conf.,
  Cleveland/Ohio 1983, Pure Appl. Math., Marcel Dekker 98, 1-19 (1986).}, 1986.

\bibitem[Bur81]{burkholder:1981}
D.~L. Burkholder.
\newblock A {G}eometrical {C}haracterization of {B}anach {S}paces in which
  {M}artingale {D}ifference {S}equences are {U}nconditional.
\newblock {\em Annals of Probability}, 9(6):997--1011, 1981.

\bibitem[CWW85]{chang_wilson_wolff:1985}
S.Y.A. Chang, J.M. Wilson, and T.H. Wolff.
\newblock {Some weighted norm inequalities concerning the Schr\"odinger
  operators.}
\newblock {\em Comment. Math. Helv.}, 60:217--246, 1985.

\bibitem[Dav80]{davis:1980}
Burgess Davis.
\newblock Hardy spaces and rearrangements.
\newblock {\em Trans. Amer. Math. Soc.}, 261(1):211--233, 1980.

\bibitem[Fig88]{figiel:1988}
T.~Figiel.
\newblock {O}n {E}quivalence of {S}ome {B}ases to the {H}aar {S}ystem in
  {S}paces of {V}ector-valued {F}unctions.
\newblock {\em Bulletin of the Polish Academy of Sciences}, 36(3--4):119--131,
  1988.

\bibitem[Fig90]{figiel:1990}
T.~Figiel.
\newblock Singular {I}ntegral {O}perators: {A} {M}artingale {A}pproach.
\newblock In {\em Geometry of Banach Spaces}, number 158 in London Mathematical
  Society Lecture Note Series, pages 95--110, 1990.

\bibitem[FW01]{figiel_wojtaszczyk:2001}
T.~Figiel and P.~Wojtaszczyk.
\newblock Special bases in function spaces.
\newblock In {\em Handbook of the geometry of {B}anach spaces, {V}ol. {I}},
  pages 561--597. North-Holland, Amsterdam, 2001.

\bibitem[GJ82]{garnett_jones:1982}
John~B. Garnett and Peter~W. Jones.
\newblock {BMO from dyadic BMO.}
\newblock {\em Pac. J. Math.}, 99:351--371, 1982.

\bibitem[Hyt11]{hytonen:2011}
T.~P. Hytonen.
\newblock Foundations of vector-valued singular integrals revisited---with
  random dyadic cubes.
\newblock {\em http://arxiv.org/abs/1110.5826, [v1] Wed, 26 Oct 2011}, 2011.

\bibitem[Kah85]{kahane:1985}
J.-P. Kahane.
\newblock {\em Some Random Series of Functions, Second Edition}.
\newblock Cambridge University Press, 1985.

\bibitem[Lec11]{lechner:2011}
R.~Lechner.
\newblock An {I}nterpolatory {E}stimate and {S}hift {O}perators.
\newblock {\em Ph.D. thesis, July 2011,
  http://shrimp.bayou.uni-linz.ac.at/Papers/dvi/phd\_thesis\_Richard\_Lechner.pdf},
  2011.

\bibitem[MP11]{mueller_passenbrunner:2011}
Paul F.~X. M{\"u}ller and Markus Passenbrunner.
\newblock A decomposition theorem for singular integral operators on spaces of
  homogeneous type.
\newblock {\em To appear in the Journal of Functional Analysis}, 2011.

\bibitem[M{\"u}l05]{mueller_p.f.x.:2005}
Paul F.~X. M{\"u}ller.
\newblock {\em {Isomorphisms between $H^1$ spaces.}}
\newblock {Monografie Matematyczne. Instytut Matematyczny PAN (New Series) 66.
  Basel: Birkh\"auser.}, 2005.

\bibitem[NS97]{novikov_semenov:1997}
I.~Novikov and E.~Semenov.
\newblock {\em {Haar series and linear operators.}}
\newblock {Kluwer Academic Publishers}, 1997.

\bibitem[NTV97]{nazarov_treil_volberg:1997}
F.~Nazarov, S.~Treil, and A.~Volberg.
\newblock {Cauchy integral and Calderón-Zygmund operators on nonhomogeneous
  spaces}.
\newblock {\em Internat. Math. Res. Notices}, 15:703--726, 1997.

\bibitem[NTV03]{nazarov_treil_volberg:2003}
F.~Nazarov, S.~Treil, and A.~Volberg.
\newblock {The $Tb$-theorem on non-homogeneous spaces.}
\newblock {\em Acta Math.}, 190(2):151--239, 2003.

\bibitem[Ste70]{stein:1970-lp}
E.~M. Stein.
\newblock {\em Topics in harmonic analysis related to the {L}ittlewood-{P}aley
  theory.}
\newblock Annals of Mathematics Studies, No. 63. Princeton University Press,
  Princeton, N.J., 1970.

\bibitem[Wol82]{wolff:1982}
Thomas~H. Wolff.
\newblock {Two algebras of bounded functions.}
\newblock {\em Duke Math. J.}, 49:321--328, 1982.

\end{thebibliography}
\end{document}